\documentclass[11pt,reqno,oneside]{amsart}

\usepackage{graphicx}
\usepackage{amssymb}
\usepackage{epstopdf}
\usepackage[a4paper, total={6in, 9in}]{geometry}
\usepackage{array} 
\usepackage{overpic} 
\usepackage{tabularx}
\usepackage{caption}
\usepackage{subcaption}
\usepackage{hyperref}
\hypersetup{hypertex=true,colorlinks=true,linkcolor=blue,anchorcolor=blue,citecolor=blue}
\usepackage{amsmath,amsfonts,amsthm,mathrsfs,amssymb,cite}
\usepackage[usenames]{color}
\usepackage{xcolor}
\usepackage{overpic} 

\newtheorem{thm}{Theorem}[section]

\newtheorem{lem}{Lemma}[section]
\newtheorem{prop}{Proposition}[section]
\theoremstyle{definition}
\newtheorem{defn}{Definition}[section]

\theoremstyle{remark}
\newtheorem{rem}{Remark}[section]
\numberwithin{equation}{section}
\numberwithin{equation}{section}
\newcounter{saveeqn}

\newcommand{\norm}[1]{\left\lVert #1 \right\rVert}

\newcommand{\RR}{\mathbb{R}}

\DeclareMathOperator{\dist}{dist}

\allowdisplaybreaks

\title[Stable determination of an impedance obstacle]{Stable determination of an impedance obstacle by a single far-field measurement}

\author{Huaian Diao}
\address{School of Mathematics, Jilin University, Changchun 130012, China}
\email{diao@jlu.edu.cn, hadiao@gmail.com}

\author{Hongyu Liu}
\address{Department of Mathematics, City University of Hong Kong, Kowloon, Hong Kong SAR, China}
\email{hongyu.liuip@gmail.com, hongyliu@cityu.edu.hk}

\author{Longyue Tao}
\address{School of Mathematics and Statistics, Northeast Normal University, Changchun 130024, China}
\email{taoly931@nenu.edu.cn, sdyctly@163.com}

\begin{document}
\maketitle
%--------------------------------------------------------
\begin{abstract}

We establish sharp stability estimates of logarithmic type in determining an impedance obstacle in $\mathbb{R}^2$. 
The obstacle is polygonal shape and the surface impedance parameter can be variable. 
We establish the stability results using a single far-field pattern, constituting a longstanding problem in the inverse scattering theory. 
This is the first stability result in the literature in determining an impedance obstacle by a single far-field measurement. 
The stability in simultaneously determining the obstacle and the boundary impedance is established in terms of the classical Hausdorff distance. 
Several technical novelties and development in the mathematical strategy developed for establishing the aforementioned stability results exist. 
First, the stability analysis is conducted around a corner point in a micro-local manner. 
Second, our stability estimates establish explicit relationships between the obstacle's geometric configurations and the wave field's vanishing order at the corner point. 
Third, we develop novel error propagation techniques to tackle singularities of the wave field at a corner with the impedance boundary condition. 

\medskip
	

\end{abstract}
%-------------------------------------------------------------
\section{Introduction}
%----------------------------------------------------------------
\subsection{The direct and inverse impedance scattering problems}

We first introduce the direct scattering problem associated with an impedance obstacle due to an incident plane wave. 
Let $\Omega\subset \RR^{2}$ be an impenetrable obstacle situated in a homogeneous background medium, where $\Omega$ is a bounded Lipschitz domain with a connected complement $\RR^{2}\setminus\overline{\Omega}$.
Let $k=\omega/c \in \RR^+$ be the wavenumber of a time-harmonic plane \textit{incident wave}
\begin{equation}\label{eq incident wave}
	u^i(\mathbf x)=e^{\mathrm{i} k \mathbf p\cdot \mathbf x},
\end{equation}
with $\omega \in \RR^+$, $c \in \RR^+$ and $\mathbf p\in \mathbb{S}^{1}$, respectively, signifying the frequency, sound speed, and the propagation direction, and $\mathrm{i}:=\sqrt{-1}$ being the imaginary unit.
Due to the interaction between the incident wave and the obstacle $\Omega$, the \textit{scattered wave} is generated, which is denoted by $u^s(\mathbf x)$.
Let $u(\mathbf x):=u^i(\mathbf x)+u^s(\mathbf x)$ be the \textit{total wave field}. 
The direct scattering problem is described by the following exterior impedance boundary value Helmholtz system:
\begin{equation}\label{impedance problem}
	\begin{cases}
	&\Delta u+k^2u=0  \hspace*{3.1cm} \mbox{in}\quad \mathbb R^2\setminus \Omega ,\medskip\\
	& u=u^i+u^s \hspace*{3.5cm} \mbox{in}\quad \mathbb R^2\setminus \Omega,\medskip\\
	&\frac{\partial u}{\partial \nu}+\eta  u=0   \hspace*{3.4cm} \mbox{on}\quad \partial \Omega , \medskip\\
	&\lim\limits_{r\to \infty}r^{1/2}(\partial _ru^s-\mathrm i ku^s)=0\hspace*{1.1cm} r=\vert \mathbf x\vert,
	\end{cases}
\end{equation}
where $\nu$ denotes the exterior unit normal vector of $\partial \Omega $ and $\eta \in L^\infty(\partial \Omega)$ signifies the surface impedance parameter.

The boundary condition of \eqref{impedance problem} is the \textit{impedance boundary condition} (also known as the \textit{Leontovitch boundary condition} or \textit{Robin boundary condition}) used to model imperfectly conducting obstacles, where $\eta \in L^\infty (\partial \Omega )$ is a given complex-valued function with $\Re\eta\geq 0$.
There are two important special cases when $\eta \equiv 0$ and we have a \textit{Neumann boundary condition}, and when $\eta=\infty$, which we interpret as a \textit{Dirichlet boundary condition} on $\partial\Omega $, namely $u=0$ on $\partial\Omega$.
Borrowing terminology from acoustics, these are the \textit{sound-hard} and \textit{sound-soft} cases respectively.
More details on the impedance boundary conditions and the generalized impedance boundary conditions can be found in \cite{colton2019inverse}. 
It is noted that the two-dimensional Helmholtz system \eqref{impedance problem} can also be used to describe the transverse electromagnetic scattering due to an impedance conductor (cf.\cite{cao2020determining}). 
To simplify the exposition, we stick to the terminologies in the acoustic scattering setup throughout the rest of the paper. 

The existence and uniqueness of the solution $u(\mathbf x)\in H^1_{\rm loc}(\mathbb R^2 \backslash \Omega )$ for the exterior impedance problem \eqref{impedance problem} can be found in \cite{liu2012singular,mclean2000strongly}. 
Throughout the rest of this paper, $u(\mathbf x)$ is said to be the solution to \eqref{impedance problem} associated with $(\Omega, \eta)$, which indicates that the impedance obstacle and corresponding impedance parameter on the boundary of \eqref{impedance problem} are $\Omega$ and $\eta$, respectively. 

The last formula in \eqref{impedance problem} represents the scattering wave $u^s(\mathbf x)$ satisfying the \textit{Sommerfeld radiation condition}.
Hence the scattered waves have the following asymptotic expansion:
\begin{equation}\label{eq:far-field parton}
	u^s(\mathbf x)=\frac{e^{\mathrm i k\vert \mathbf x\vert }}{\vert \mathbf x\vert^{1/2} }\left(u^\infty(\hat{\mathbf x})+\mathcal O\left(\frac{1}{\vert \mathbf x\vert^{1/2}}\right)\right), 
	\quad \vert \mathbf x\vert \to \infty
\end{equation}
uniformly in all directions $ \hat{\mathbf x}=\mathbf x/|\mathbf x|$, where the complex valued function  $u^{\infty}(\hat{\mathbf x})$ in \eqref{eq:far-field parton} defined over the unit sphere $\mathbb{S}^1$ is known as the \textit{far-field pattern} of $u^s(\mathbf x)$ with $\hat{\mathbf x} \in \mathbb{S}^1$ signifying the observation direction.

The inverse problem is concerned with the recovery of the position, shape, and impedance parameter of an obstacle by the far-field measurement, which can be formulated by
\begin{align}\label{eq:IP}
	\mathfrak{F}(\Omega,\eta)=u^{\infty}(\hat{\mathbf{x}};u^i),\ \ \hat{\mathbf{x}}\in \mathbb{S}^1,
\end{align}
where $\Omega$ is the underlying obstacle in \eqref{impedance problem} and  $u^{\infty}(\hat{\mathbf{x}})$ is the far-field pattern associated with \eqref{impedance problem}. 
The inverse problem \eqref{eq:IP} is known to be nonlinear and ill-posed. 
This paper is concerned with the stability analysis for a polygonal impedance obstacle in $\mathbb R^2$ by a single far-field measurement. 
That is, in the inverse problem \eqref{eq:IP}, the far-field pattern is collected corresponding to a given and fixed $u^i(\mathbf x)$. 

%----------------------------------------------------------------
\subsection{Statement of the main stability results}

Let us first define some basic notations for our subsequent study. 
For any $\mathbf x\in \mathbb{R}^{2}$ and $r\in \mathbb R_+$, we denote $B_r( \mathbf x)$ by the \textit{open disk} with the center $ \mathbf x$ and radius $r$.
When $\mathbf x=\mathbf 0$, $B_r$ stands for $B_r(\mathbf 0)$ for any $r>0$.
For any open bounded Lipschitz domain $D\subset \mathbb{R}^{2}$, $B_r(D)=\cup_{\mathbf x\in D}B_r(\mathbf x)$.
In addition, ${\mathrm{diam}}(D)$ signifies the diameter of $D$.

For any fixed ${\omega_0} \in \mathbb S^1$, $\mathbf x\in \mathbb{R}^{2}$, and two given constants $r_0\in \mathbb R_+$, $\theta_0\in (0,\pi)$, $\mathcal{C}$ is the \textit{open cone} defined by
	$$
    \mathcal{C}(\mathbf x,{\omega_0},r_0,\theta_0)=\{\mathbf y\in \mathbb{R}^{2}~|~0<|\mathbf y-\mathbf x|<r_0,\quad \angle(\mathbf y-\mathbf x,\omega_0 )\in [0,\theta_0/2 ]\}.
    $$
Here, $\mathbf x$ is the vertex of the open cone $\mathcal{C}$, $r_0$ is its radius, ${\omega_0}$ is the axis of $\mathcal{C}$, and $\theta_0 $ is its amplitude angle.

To investigate the direct and inverse acoustic obstacle scattering problems, the distance between two obstacles and the convergence of obstacles will be introduced as follows, where we utilize the \textit{classical Hausdorff distance}.
Assuming that $K$ and $K'$ are two obstacles, we use $d_H(K, K')$ to denote the classical Hausdorff distance between these two obstacles, which is defined as follows
\begin{equation}\label{eq:hauss}
   d_H(K,K')=\max \left\{ \sup_{\mathbf x\in K}\dist(\mathbf x,K'),\sup_{\mathbf x\in K'}\dist(\mathbf x,K) \right\}.
\end{equation}
It is ready to see that $d_H(K, K')=0$ if and only if the two obstacles $K$ and $K'$ are the same in terms of position and shape.
We establish a stability estimation for two polygonal obstacles concerning their corresponding far-field measurements error.

Now we introduce the admissible class of obstacles for our study. 
Let us write $G:=\mathbb{R}^{2}\setminus K$ as the complementary set of a set $K$.

\begin{defn}\label{def:Class B}
	We say a obstacle $K$ belongs to the \textit{admissible class} $\mathcal{B}$ with some a-priori parameters $(\underline{\ell} ,\ \overline{\ell},\ \underline{\theta},\ \overline{\theta},R,r_m,\delta,\theta)$, if $K$ satisfies the following assumptions
	\begin{itemize}
		\item [(1)] $K\Subset \overline{B_R(\mathbf 0)}$;
        \item [(2)] the obstacle $K\subset \RR^2$ is a bounded open convex polygon;
		\item [(3)] each edge length $\ell$  of $K$ satisfies $\ell\in [\underline{\ell},\overline{\ell}]$, where $\underline{\ell}$ and $\overline{\ell}$ are positive constants;
		\item [(4)] each vertex interior opening angle $\theta$ of $K$ satisfies  $  \theta \in [\underline{\theta},\overline{\theta}]$, where $\overline \theta  $ and $ \underline \theta$  are constants with $0<\underline{\theta}<\overline{\theta}<2\pi$;
		\item [(5)] for any $r$ satisfying $0<r<r_m$, where $r_m\in \mathbb R_+$ is fixed, $\RR^{2}\setminus \overline {B_{r}(K)}$ is connected;
		\item [(6)] for any $\mathbf x\in \partial K$, there exists a direction ${\omega_0} \in \mathbb{S}^{1}$, as $\mathbf y\in \partial K\cap B_{\delta}(\mathbf x)$, the open cones $\mathcal{C}(\mathbf y,\mathbf{\omega_0},\delta,\theta)$ are contained in $G$ or their opposite cones $\mathcal{C}(\mathbf y,\mathbf{-\omega_0},\delta,\theta)$ are all contained either in $K$ or in $G$;
		\item [(7)] $K$ is a Lipschitz domain and an obstacle;
\end{itemize}
\end{defn}

\begin{rem}\label{rem:31}
When the underlying obstacle $K$ associated with \eqref{impedance problem} satisfies $K \in\mathcal{B}$, according to Definition~\ref{def:Class B}, one knows that $K$ fulfills the uniform cone condition and it has Lipschitz boundary with local compactness property.
By \cite[Theorem 1]{ramm1996existence}, there exists a unique $H^1$ weak solution for the exterior impedance problem \eqref{impedance problem}, and satisfies that the embedding $H^1(\mathbb{R}^2\setminus K ) \rightarrow L^2(\mathbb{R}^2\setminus K)$ is compact and the trace operator $H^1(\partial K )\ \rightarrow L^2(\partial K )$ is compact.
\end{rem}

\begin{defn}\label{defn: eta}
We say a impedance parameter $\eta$ belongs to the \textit{admissible class } $\Xi_{\mathcal B}$ with some priori parameters $(M_1,M_2,\alpha_0,C_{\mathcal B,\ \Xi_{\mathcal B}})$, if $\eta$ satisfies the following assumption:
    	
Let $M_1$, $M_2$, and $\alpha_0$ ($0<\alpha_0\leq 1$) be given positive constants.  
For any $K \in \mathcal B$, it holds that
\begin{align}\label{eq:xi}
     \eta\in \Xi_{\mathcal B}:=	&\{ \eta (\mathbf x) \in C^1(\partial K)~| ~ 0 \le |\eta (\mathbf x) |\le M_1,\\
     &\left| \eta (\mathbf x) - \eta (\mathbf y) \right| \le M_2 \left| \mathbf x - \mathbf y \right|^{\alpha _0 },\ \forall \mathbf x,\mathbf y\in \partial K\}.\notag
\end{align}
\end{defn}

\begin{rem}\label{rem:13}
It is clear that when the impedance parameter $\eta$ is constant, then it fulfills \eqref{eq:xi}. 
Especially, if $\eta\equiv 0$, an impedance obstacle degenerates to be a sound-hard obstacle. 
\end{rem}

Throughout the rest of the paper, let us fix two obstacles $K\in \mathcal{B}$ and $K'\in \mathcal{B}$. 
We also fix the incident wave $u^i(\mathbf x)$ given by \eqref{eq incident wave}. 
That is the wave number $k\in \mathbb R_+$ and incident direction $\mathbf p \in \mathbb{S}^{1}$ are fixed. 
Let $u(\mathbf x)$ be the solution of \eqref{impedance problem} associated with $K$ and the impedance parameter $\eta$. 
Similarly, $u'(\mathbf x)$ is the total wave field associated with the obstacle $K'$ to \eqref{impedance problem} corresponding to the impedance boundary condition with an impedance parameter $\eta'$. 
Homoplastically, we denote $u^{s}(\mathbf x)$ and $(u^{s})'(\mathbf x)$ by the corresponding scattered fields, and their far-field patterns are $u^{\infty}_{(K,\eta)}(\hat{\mathbf x})$ and $u^{\infty}_{(K',\eta')}(\hat{\mathbf x})$ respectively.

%------------------------------------------------------------------
In Theorem~\ref{mian result}, the double-log type stability result for the shape determination of an admissible scatter $K\in \mathcal B$ by a single measurement is given, which is independent of the impedance parameter. 

\begin{thm}\label{mian result}
 Let $K,K'\in\mathcal B$ be two admissible obstacles described in Definition \ref{def:Class B}. 
 Let $u^i(\mathbf x)$ be a fixed time-harmonic plane incident wave of the form \eqref{eq incident wave}. 
 Consider the acoustic impedance obstacle scattering problem \eqref{impedance problem} associated with $K$ and $K'$, which the two impedance parameters $\eta(\mathbf x),\eta'(\mathbf x)\in\Xi_{\mathcal B}$ are associated with $K,K'$, respectively.  
 Assume that $u^\infty_{(K,\eta)}(\hat{\mathbf x})$ and ${u^{\infty }_{(K',\eta')} (\hat{\mathbf{x} } )}$ are the far-field patterns of the scattered waves $u^s(\mathbf x)$ and $(u^s)'(\mathbf x)$ associated with the impedance obstacle obstacles $K$ and $K'$, respectively. 
 If 
\begin{equation}\label{eq:far-field error}
 	{\left\| u^{\infty }_{\left( K,\eta \right)}  (\hat{\mathbf{x} } ) - {u^{\infty }_{\left( K',\eta^{\prime } \right)}(\hat{\mathbf{x} } )} \right\|_{{L^\infty }({\mathbb{S}^1})}} \leq \varepsilon, \quad \varepsilon \in \mathbb R_+,
\end{equation}
where $\varepsilon<\varepsilon_0$ with $\varepsilon_0\in\mathbb{R}^+$ being sufficiently small and depending only on the a-priori parameters, then the Hausdorff distance $\mathfrak{h}$ of $K$ and $K'$ defined in \eqref{eq:hauss} satisfies
 \begin{equation}\label{eq:stability function}
 	\mathfrak{h} \leq C(\ln \ln (1/\varepsilon))^{-\kappa},
 \end{equation}
 where $C$ and $\kappa$ are positive constants, depending on the a-prior parameters only.

\end{thm}

\begin{rem}
The parameter $\kappa$ appearing in \eqref{eq:stability function} will be specified in Section~\ref{sec:proof}. 
We emphasize that our developed proof for Theorem~\ref{mian result} can be used to prove similar stability results for the determination of sound-soft and sound-hard polygonal obstacles by a single measurement since the stability estimation \eqref{eq:stability function} is independent of the impedance parameter. 
Indeed, as discussed in Remark~\ref{rem:13}, the sound-hard obstacle can be viewed as a degenerated impedance obstacle. 
For the sound-soft polygonal obstacle, we can modify our proof of Theorem~\ref{mian result} to derive a similar stability result for the shape determination. 
\end{rem}

We also show that the impedance parameter of convex polygonal impedance obstacles can be stably determined by a single far-field measurement simultaneously. 

\begin{thm}\label{th: eta}
Let $K,K'\in\mathcal B$ be fixed admissible obstacles described in Definition~\ref{def:Class B}.
Let $u^i(\mathbf x)$ be a fixed time-harmonic plane incident wave of the form \eqref{eq incident wave}. 
Consider the acoustic impedance obstacle scattering problem \eqref{impedance problem}.  
Let $\eta$ and $\eta'$ be two non-zero constant impedance parameters associated with $K$ and $K'$, respectively.  
Suppose that $u_{\left( K,\eta \right)} =u^i+u^s_{\left( K,\eta \right)}$ and $u_{\left( K',\eta' \right)} =u^i+u^s_{\left( K',\eta' \right)}$ are total wave fields to \eqref{impedance problem} associated with  $(K,\eta)$ and $(K',\eta')$, respectively, where $u^s_{\left( K,\eta \right)} (\mathbf x)$ and $u^s_{\left( K',\eta' \right)} (\mathbf x)$ are scattered waves. 
Assume that $u^{\infty }_{\left( K,\eta \right)}  (\hat{\mathbf{x} } ),u^{\infty }_{\left( K,\eta' \right)}  (\hat{\mathbf{x} } )$ and ${u^{\infty }_{\left( K',\eta^{\prime } \right)}(\hat{\mathbf{x} } )}$ are the far-field patterns of the scattered waves $u^s_{\left( K,\eta \right)} (\mathbf x),u^s_{\left( K,\eta' \right)} (\mathbf x)$ and $u^s_{\left( K',\eta^{\prime } \right)}(\mathbf x)$ respectively. 
      	   
%Denote $H=\mathbb R^2 \setminus (K\cup K')  $ and 
%\begin{align}\notag
%C_{\mathcal B}&=\min\Big\{\sup  \limits_{\mathbf x\in \Gamma  \setminus {\mathcal V}(K)   }\left\{|u_{\left( K,\eta \right)} (\mathbf x)|~\big |~\Gamma \in \partial H\cap \partial K \right\}, \\
%&\quad \quad \quad \ \sup  \limits_{\mathbf x\in \Gamma'  \setminus {\mathcal V}(K')   }\{|u_{\left( K',\eta' \right)} (\mathbf x)|~\big   |~\Gamma' \in \partial H\cap \partial K'\}\Big\},\notag
%\end{align}
%where ${\mathcal V}(K) $ is the set of the vertices of $K$.  
If 
\begin{equation}\label{eq: new far error}
 	{\left\| u^{\infty }_{\left( K,\eta \right)}  (\hat{\mathbf{x} } ) - {u^{\infty }_{\left( K',\eta^{\prime } \right)}(\hat{\mathbf{x} } )} \right\|_{{L^\infty }({\mathbb{S}^1})}} \leq \varepsilon, \quad \varepsilon \in \mathbb R_+,
\end{equation}
where $\varepsilon<\varepsilon_0$ with $\varepsilon_0\in\mathbb{R}^+$ being sufficiently small and depending only on the a-priori parameters, then one has
\begin{equation}\label{eq: final eta}
   | \eta   -\eta^{\prime }| \leq \psi(\varepsilon),
\end{equation}
where 
$$
  \psi(\varepsilon):= C_P  \left\{\ln\left|\ln \left[\exp\left(-C_a(-\ln \varepsilon)^{1/2}\right)+ {\frac
    	 {C}{R}\left(\ln \ln \frac{1}{\varepsilon}\right)^{-\varsigma \kappa}}\right]\right|\right\}^{-\alpha}. 
$$
Here $C$, $C_a$, $C_P$, $\varsigma$, $\alpha$, and $ \kappa$ are positive constants, depending on the a-prior parameters only.
\end{thm}

%\begin{rem}\label{rem:15}
%	We require the positiveness of the a-prior parameter $C_{\mathcal B}$ in Theorem \ref{th: eta}.   
%	According to the regularity result for the impedance scattering problem \eqref{impedance problem} discussed in \cite{clm2012}, where the underlying obstacle is a polygon $K$, the corresponding total wave field is continuous up to the boundary of the obstacle. 
%	By using the unique continuation and radiation condition of the scattered field, we know that the total wave field cannot vanish on any open subset of $\partial K$ (cf. \cite{cao2020nodal}). 
%	Therefore we conclude that the a-prior parameter  $C_{\mathcal B}$ in Theorem~\ref{th: eta} satisfies that $C_{\mathcal B}>0$. 
%\end{rem}

\begin{rem}
When the impedance parameter $\eta$ is a nonzero constant, we have proved the stability result \eqref{eq: final eta} for determining $\eta$ by a single measurement. 
It can be seen from the proof of Theorem~\ref{th: eta} in what follows, our argument can be extended to establishing a similar stability result in determining a variable impedance parameter. 
However, this will involve a bit tedious description of the general conditions on the variable impedance parameter. 
Hence, we skip it and instead focus on developing the technical strategy for the constant case. 
\end{rem}

The proof of Theorem~\ref{mian result} will be given in Section \ref{sec:proof}, and the proof of Theorem~\ref{th: eta} will be given in Section~\ref{sec:proof2}. 
Before those, we need to obtain the far-field error propagation behavior from far-field to near-field in Section~\ref{sec:near field} and give a micro-local analysis of corner scattering in Section~\ref{sec:micro}. 

%-----------------------------------------------------------------
\subsection{Discussions on technical novelty and development}

The inverse obstacle problem \eqref{eq:IP} with a single far-field measurement constitutes a longstanding problem in the literature (cf. \cite{liu2022local,colton2018looking,colton2019inverse}). 
A sound-soft or sound-hard polyhedral obstacle can be uniquely determined by a single far-field measurement, which has been studied extensively in the last two decades and critically relies on the reflection principle for the Helmholtz equation with Dirichlet and Neumann boundary conditions; see \cite{alessandrini2005determining,cheng2003uniqueness,liu2006uniqueness,elschner2006uniqueness} and the references cited therein. 
Local and global unique determinations of the shape and the physical boundary parameters of an impedance polyhedral obstacle by a single far-field measurement under generic conditions are studied in \cite{cao2020nodal,cao2021novel,cao2022two}, where the local geometrical properties of Laplacian eigenfunctions are utilized. 

The stability results for the determination of obstacles can be regarded as quantifications of the corresponding uniqueness results, and are usually technically more involved and challenging. 
In \cite{rondi2008stable}, stability estimates for the determination of sound-soft polyhedral obstacles in $\mathbb R^3$ with a single measurement are established. 
The stability results in the case of sound-hard polyhedral obstacles can be found in \cite{liu2017stable}. 
However, the stability estimation in \cite{rondi2008stable} and \cite{liu2017stable} is obtained by using the reflection principle for the Helmholtz equation with Dirichlet and Neumann boundary conditions, which cannot deal with the impedance boundary condition. 

In Theorems~\ref{mian result} and \ref{th: eta}, we establish stability estimates of logarithmic type in determining an impedance obstacle in $\mathbb{R}^2$. 
It is known in the inverse scattering theory that stability estimates of the logarithmic type are generically sharp. 
To our best knowledge, these are the first stability results in the literature in determining an impedance obstacle by a single far-field measurement. 
The stability in simultaneously determining the obstacle and the boundary impedance is established in terms of the classical Hausdorff distance. 
Furthermore, for the stability investigation of the surface impedance parameters of impedance obstacles one can refer to \cite{sincich2006stable,alessandrini2013stable} that established stability estimates in logarithmic form. We point out that our stability results establish the stability of the impedance parameter simultaneously with the stability of the obstacle shape and position.

Finally, we briefly mention several technical novelties and development in the mathematical strategy that we develop for establishing the aforementioned stability results. 
First, the stability analysis is conducted around a corner point in a micro-local manner. 
Second, our stability estimates establish explicit relationships between the obstacle's geometric configurations and the wave field's vanishing order at a certain corner point. 
In particular, it is revealed that the larger the vanishing order is, the less stable reconstruction one can expect on the inverse problem. 
This is consistent with the physical intuition and more detailed discussions can be found in Subsection~\ref{subsec:51}. 
Third, we develop novel error propagation techniques to tackle singularities of the wave field at a corner as well as to tackle the impedance boundary condition. 
In principle, our arguments can be extended to the three-dimensional setting but shall require several new technical extensions. Hence, we focus on the two-dimensional setup in the current study to develop our method and shall consider the three-dimensional extension in a forthcoming paper. 

The rest of the paper is organized as follows. 
In Section~\ref{sect:2}, we derive the uniform boundedness of the wave field within an admissible class based on the Mosco convergence for the impedance obstacle scattering problem, which is of independent interest to the literature. 
In Section~\ref{sec:near field}, we propagate the error from the far-field to the near-field and then to the boundary using the three-sphere inequality as well as the local H\"older continuity. In Section~\ref{sec:micro}, we conduct quantitative micro-local analysis around a corner. 
Section~\ref{sec:proof} is devoted to the proof of the main stability result as well as some related discussions. 

%-------------------------------------------------------------------
\section{Mosco convergence for the impedance obstacle scattering problem and uniformly bounded estimations}\label{sect:2}

Consider the impedance obstacle scattering problem \eqref{impedance problem}, let us first discuss important properties regarding uniformly bounded estimates of the total wave field $u$ associated with an admissible obstacle $K\in \mathcal B$, where we utilize Mosco convergence for the impedance obstacle scattering problem \eqref{impedance problem} concerning the admissible class $\mathcal B$.

%-------------------------------------------------------------------
\subsection{Mosco convergence for the impedance obstacle scattering  problem}\label{subsec:mosco}

We give the compactness of the admissible class $\mathcal{B}$ in Proposition~\ref{prop:compact} in terms of the Hausdorff distance, which will be used to discuss the Mosco convergence for the impedance obstacle scattering problem. 
%We omit the detailed proof of Proposition~\ref{prop:compact}. 

\begin{prop}\label{prop:compact}
	
The admissible obstacle  class $\mathcal{B}$ is compact with respect to the Hausdorff distance.
\end{prop}

\begin{proof}
	The properties of the admissible obstacles class $\mathcal B$ are closed with respect to the Hausdorff distance, and the details can be found in \cite[Section 2]{rondi2008stable}.
	Furthermore, under the assumption of admissible classes of surface impedance parameters, the Arzela-Ascoli Theorem will guarantee the $\Xi_\mathcal B$ is a compact set too.
	\end{proof}

About the admissible obstacle class $\mathcal B$, by Proposition~\ref{prop:compact}, $\mathcal B$ is compact with respect to Hausdorff distance.
According to Remark~\ref{rem:31}, $\forall K\in\mathcal B$, the direct problem \eqref{impedance problem} associated with the obstacle $K$ has a unique weak solution in $H^1_{\rm loc}( \mathbb R^2\backslash K )$.
To establish a uniformly bounded estimate of the solution to \eqref{impedance problem} for $\mathcal B$ when the associated obstacle $K$ belongs to $\mathcal B$, we introduce the definition of \textit{Mosco convergence} (cf.\cite{mosco1969convergence}).

Mosco convergence for the sound-hard case of \eqref{impedance problem} was introduced in \cite{menegatti2013stability}, namely the boundary condition in \eqref{impedance problem} satisfies Neumann boundary condition.
In this subsection, we establish Mosco convergence for the impedance obstacle scattering problem \eqref{impedance problem}.
Let us introduce Mosco convergence on reflexive Banach space and the equivalent definition on Sobolev space (cf.\cite{bucur2000boundary}).

\begin{defn}\label{def:31}
	Given $\{S_n\}_{n\in \mathbb{N}}$, a closed subset sequence of a reflexive Banach space $X$, we denote
\begin{align*}
	& S'=\{x\in X: x=w-\displaystyle\lim_ {k\rightarrow \infty}x_{n_{k}} ,x_{n_k} \in S_{n_k} \},\\
	& S''=\{x\in X: x=s-\displaystyle\lim _{n\rightarrow \infty}x_{n} ,x_{n} \in S_{n} \},
\end{align*}
where $s-\displaystyle\lim_{n\rightarrow \infty}$ represents strong convergence for all $n\in \mathbb N$ and $w-\displaystyle\lim_{k\rightarrow \infty}$ represents weak convergence for subsequence.
It's said that $S_n$ \textit{converges in the sense of Mosco} to $S$, if we have $S=S{'}=S{''}$.
\end{defn}

Throughout this section, In addition, we set $\Omega=B_{R+1}$. 
Let $u_n(\mathbf x)\in H^1(\Omega \setminus K_n)$ be the weak solution to \eqref{impedance problem} associated with $(K_n,\eta_n)$, where $K_n\in \mathcal B$ are admissible obstacles and $\eta_n\in\Xi_{\mathcal B}$ is restricted to $\partial K_n$. 

\begin{defn}
Given a sequence $\{K_n\}_{n\in \mathbb{N}}$ belonging  to the admissible obstacle class $\mathcal{B}$, let $K\in \mathcal B$ be a fixed obstacle.
We say that the sequence $\{K_n\}_{n\in \mathbb{N}}$ is a \textit{K-convergence obstacle sequence with the  convergent point $K$}, if the sequence $\{K_n\}_{n\in \mathbb{N}} \in \mathcal{B}$ converges to $K \in \mathcal{B}$ with respect to Hausdorff distance.
\end{defn}

By Definition~\ref{def:31}, if the weak solution to \eqref{impedance problem} associated with $(K_n,\eta_n)$ exists, Mosco convergence in Sobolev space is introduced in the following definition.

\begin{defn}\label{def:mosco in sobolev space}
Let $K\in \mathcal B$ be a fixed obstacle.
Suppose that a sequence of obstacles $\{K_n\}_{n\in \mathbb{N}}$ in the admissible obstacle class $\mathcal{B}$ is a K-convergence obstacle sequence with the convergent point $K$.
Furthermore, assume that $u_n(\mathbf x) \in H^1(\Omega\setminus K_n)$ is the weak solution to \eqref{impedance problem} associated with $(K_n,\eta_n)$.
Let $u(\mathbf x) \in H^1(\Omega\setminus K)$  be a weak solution of  \eqref{impedance problem} associated with $(K,\eta)$, where $\eta,\eta_n\in\Xi_{\mathcal B}$ are admissible impedance parameters. 
It is said that $H^1(\Omega\setminus K_n)$ \textit{converges in the sense of Mosco} to $H^1(\Omega\setminus K)$ if the following two conditions are equivalent:
\begin{enumerate}
		\item for $u\in H^1(\Omega\setminus K)$, the sequence of functions $u_n\in H^1(\Omega\setminus K_n) $ converges strongly in $L^2(\Omega)$ to $u$ and $\nabla u_n$ converges strongly in $L^2(\RR^2)$ to $\nabla u$;
		\item for every sequence $u_{n_k}\in H^1(\Omega\setminus K_{n_k})$ such that $(u_{n_k},\nabla u_{n_k})$ is weakly convergent in $L^2(\Omega,\mathbb{R}^{2})$ to $(v,V)$, we have that $u=v$  and $\nabla u=V$ in $\Omega \setminus K$, which are understood in the sense of distribution.
\end{enumerate}
\end{defn}

If the sequence $\{K_n\}_{n\in \mathbb{N}}$ is a K-convergence sequence of the obstacle, we want to study whether the Sobolev space $H^1(\Omega\setminus K_n)$ converges to $H^1(\Omega\setminus K)$ in the sense of Mosco or not.
Furthermore, we shall investigate whether the corresponding solution $u_n(\mathbf x)$ to \eqref{impedance problem} converges to $u(\mathbf x)$ or not.
The corresponding developments regarding the aforementioned questions for sound-hard obstacle scattering problems have been shown in \cite{menegatti2013stability} and \cite{bucur2000boundary}.
Many sufficient conditions and some relevant discussions are given in \cite{menegatti2013stability}.
Therefore, in the following proposition, we investigate the aforementioned issue for impedance obstacle scattering problems \eqref{impedance problem}. 
More detailed discussions on Mosco convergence for the impedance obstacle scattering problem can be found in Remark~\ref{rem:mosco}. 
Proposition~\ref{prop mosco} is an important ingredient for proving the uniform bounded estimation of a solution to \eqref{impedance problem} in Proposition~\ref{uniform L2 bound}, where the obstacle belongs to the admissible class $\mathcal B$.  

\begin{prop}\label{prop mosco}
Let $\{K_n\}_{n\in \mathbb{N}}\in \mathcal B$ be a sequence of K-convergence obstacles, where $K\in \mathcal B$ is the limiting set, and we assume that $H^1(\Omega\setminus K_n)$ converges in the sense of Mosco to $H^1(\Omega\setminus K)$.
Let $u_n(\mathbf x) \in H^1(\Omega\setminus K_n)$ and $u(\mathbf x)\in H^1(\Omega\setminus K)$ be the weak solutions to \eqref{impedance problem} associated with $(K_n,\eta_n)$ and $(K,\eta)$, respectively, where $\eta,\eta_n\in\Xi_{\mathcal B}$ are admissible impedance parameters.  
If there exists a constant $C_0>0$, for $\forall n\in \mathbb{N}$, $\Vert u_n\Vert_{L^2(B_{R+1}\setminus K)} \leq C_0$, where $R$ is sufficient large such that $K\Subset B_{R+1}$.
Then, up to a subsequence, we have that $u_n(\mathbf x)$ is converging strongly in $L^2(\Omega)$ to $u(\mathbf x)$ as $n\rightarrow\infty$.
\end{prop}

\begin{proof}
First, we prove that the convergence point $u(\mathbf x)$ of $u_n(\mathbf x)$ still satisfies the Helmholtz equation and the Sommerfeld radiation condition without considering the boundary conditions.
By \cite[Lemma 3.1]{rondi2003unique}, up to a subsequence, we have that $u_n(\mathbf x)$ converges uniformly on compact subset of $\mathbb R^2\setminus K$ to a function $u(\mathbf x)$ with $u(\mathbf x)$ solving
\begin{equation}\label{eq:pro31}
	\begin{cases}
		&{\Delta u+k^2u=0  \hspace*{3.1cm} \mbox{in}\quad \mathbb R^2\setminus K},\medskip \\
		&{ u=u^i+u^s \hspace*{3.5cm} \mbox{in}\quad \mathbb R^2\setminus K}, \medskip\\
		&{\lim\limits_{r\to \infty}r^{1/2}(\partial _ru^s-\mathrm{i}ku^s)=0\hspace*{1.05cm} r=\vert \mathbf x\vert}.
	\end{cases}
\end{equation}

In what follows, we only need to prove that $u(\mathbf x)$ also satisfies the impedance boundary condition on $\partial K$.
Recalling $B_{R+1}:=\Omega$, without loss of generality, we assume that $u_n(\mathbf x)$ converges to $u(\mathbf x)$ weakly in $L^2(\Omega)$ and $u(\mathbf x) \in H^1(\Omega \setminus K)$ is a weak solution of \eqref{eq:pro31}.

To investigate the boundary condition, we discuss the regularity of the solution.
Let us fix an open subset $D$ with a $C^1$ boundary such that $D\Subset \Omega \backslash \Omega_1$, where $\Omega_1:=B_R \Subset \Omega$.
Let $r>0$ be sufficient small such that $B_r(\partial D)\Subset \Omega\setminus\overline{\Omega}_1$.

By the elliptic interior regularity (cf. \cite[Section 6.3, Theorem 2]{evans2022partial}), there exists a constant $C_1$ such that  $\Vert u_n \Vert_{H^3(B_r(\partial D))} \leq C_1\Vert u_n \Vert_{L_2(\Omega \setminus \overline{\Omega}_1 )}$.
Moreover, we use the Sobolev embedding theorem, we have a constant $C_2$ that satisfies $\Vert u_n \Vert_{C^{1,1}(B_{r/2}(\partial D))} \leq C_2 \Vert u_n \Vert_{H^3(B_r(\partial D))}$.
According  to $u_n(\mathbf x) \in H^1(\Omega \setminus K_n)$, we can conclude there exists  a constant $C_3$  such that
\begin{align}\label{eq:37 uc1}
	\mbox{$\Vert u_n \Vert_{C^{1}(B_{r/2}(\partial D))} \leq C_3$ for $\forall n \in \mathbb{N}$.}
\end{align}

By applying the Arzela-Ascoli theorem in \cite[Theorem~2.15]{mclean2000strongly}, we can deduce that $\Xi_{\mathcal B}$ is a compact set.
We have supposed that the impedance parameter $\eta(\mathbf x)$ associated with $K$ fulfills $\eta(\mathbf x) \in \Xi_{\mathcal B}$ defined in \eqref{eq:xi}. 
If the impedance parameter $\eta_n(\mathbf x)$ associated with $K_n\in \mathcal B$ satisfies $\mathop {\lim }\limits_{n \to \infty } {\eta _n} = \eta$  and $\eta_n(\mathbf x)\in \Xi_{\mathcal B}\subset C^1(\Omega)$, $\forall n \in \mathbb N$, where the $\eta(\mathbf x),\eta_n(\mathbf x)$ are restricted to $\partial K,\partial K_n$ respectively.

For sufficient large $n$, one has $K_n\Subset D$.
According to the weak formula of  \eqref{impedance problem}, choosing $u_n(\mathbf x)$ as the test function, it yields that
\begin{equation*}
	\int_{D/{K_n}} {|\nabla {u_n}{|^2}\mathrm d\mathbf x - {k^2}} \int_{D/{K_n}} {|u_n|^2} \mathrm d\mathbf x = \int_{\partial D} {{{\partial {u_n}} \over {\partial \nu }}}  \cdot \overline{u_n}\mathrm d\sigma - \int_{\partial {K_n}} \eta_n {|u_n|^2} \mathrm d\sigma.
\end{equation*}
By \cite[Theorem 3.37 and Theorem 3.38]{mclean2000strongly}, we can apply the trace theorem on the boundary of $K_n \in \mathcal{B}$, which implies that $\Vert u_n \Vert_{H^{1/2}(\partial K_n)} \leq C\Vert u_n \Vert_{H^1(D\setminus K_n)}$ with a constant $C>0$.
Furthermore, using \eqref{eq:37 uc1} and Cauchy-Schwartz inequality, we can get that $(u_n, \nabla u_n)$ is uniform bounded in $L^2(D,\mathbb{R}^{2})$.
Hence, there exists a subsequence (still denoted with the same index) which weakly converges to $(v, V)$ in $L^2(D,\mathbb{R}^{2})$.
By Definition~\ref{def:mosco in sobolev space}, it is not difficult to see $u=v,\ u\in H^1(D\setminus K)$ and $V=\nabla u$ in the sense of distributions.

Taking $\varphi(\mathbf x) \in H^1(\Omega \setminus K)$ such that $D$ is a compact support set of $\varphi(\mathbf x)$, by Mosco convergence we can find $\{\varphi_n(\mathbf x) \in H^1(\Omega \setminus K_n)\}_{n\in\mathbb N}$ such that $D$ also is a compact support set of $\varphi_n(\mathbf x)$ for any $n\in \mathbb{N}$, that strongly convergence to $\varphi(\mathbf x)$ as $n\rightarrow \infty$, we have

\begin{equation}\notag 
	\begin{aligned}
		\int_{\Omega/{K_n}} {\nabla {u_n} \nabla {\varphi _n}\mathrm d\mathbf x - {k^2}} \int_{\Omega/{K_n}} {{u_n}{\varphi _n}} \mathrm d\mathbf x
		= \int_{\partial {K_n}} {{\frac{\partial  u_n}{\partial \nu}}} {\varphi _n}\mathrm d\sigma
		= -\int_{\partial {K_n}} \eta_n {u_n}{\varphi
		_n}{\mathrm{d}}{\sigma}.
	\end{aligned}
\end{equation}
Since
\begin{equation*}
	\int_{\partial {K_n}} \eta_n{{u_n}} {\varphi _n}\mathrm d\sigma = \int_{\partial {K_n}\backslash K} \eta_n{{u_n}{\varphi _n}\mathrm d\sigma + \int_{\partial K} \eta_n{{u_n}{\varphi _n}\mathrm d\sigma} },
\end{equation*}
and by the trace theorem the $H^1$ norm of $\varphi(\mathbf x)$ and $u(\mathbf x)$ is finite.
Passing to the limit $\partial K_n \setminus \partial K$ is a zero measure set.
Taking the limit, the following integral equality holds
\begin{equation}\notag 
	\int_{\Omega/{K}} {\nabla {u} \cdot \nabla {\varphi }\mathrm d\mathbf x -
	{k^2}} \int_{\Omega/{K}} {{u}{\varphi }} \mathrm d\mathbf x = -\mathrm
	\int_{\partial {K}} \eta{u} {\varphi }\mathrm d\sigma.
\end{equation}
By finding an appropriate shape for $D$, we can obtain that up to a subsequence, $u_n(\mathbf x)$ converges weakly in $L^2(\Omega)$ to $u(\mathbf x)$ which satisfies the impedance boundary conditions on $\partial K$.
Due to the existence and uniqueness of the solution to \eqref{impedance problem} associated with $(K,\eta)$, we have that $u_n(\mathbf x)$ converges uniformly on compact subsets of $\mathbb R^2\setminus K$ to $u(\mathbf x)$ for $\forall n\in \mathbb N$.

Finally, we want to prove that $u_n(\mathbf x)$ strongly converges to $u(\mathbf x)$ in $L^2(B_{R+1})$.
We have assumed that for any $n\in \mathbb{N}$, $\Vert u_n\Vert_{L_2(B_{R+1}\setminus K)} \leq C_0$, and we know $u(\mathbf x)\in H^1(B_{R+1}$), therefor we can easily get $\Vert (u_n,\Delta u_n)\Vert_{L_2(B_{r},\mathbb{R}^{3})} \leq C_4$ for a fixed $r$, $R<r<R+1$, where $C_4>0$ is a constant.
Recall the assumption (6) defined by Definition~\ref{def:Class B} that the obstacle $K\in \mathcal B$ satisfies the uniform cone condition.
In \cite[Theorem 5.4]{adams1975sobolev}, if $K$ is an obstacle satisfying the cone condition with a bounded not empty open cone $\mathcal C$, then there exist $p>2$, constant $C_5>0$, then
$$\Vert u\Vert_{L_p(B_{R+1}\setminus K)}\leq C_5\Vert u\Vert_{H_1(B_{R+1}\setminus K)}\ \text{for} \ \forall u\in H_1(B_{R+1}\setminus K).$$
Thus, the required conditions for in \cite[Proposition 2.9]{menegatti2013stability} are satisfied, and since the proof of the proposition does not involve boundary conditions satisfied by $u(\mathbf x)$, we can apply this proposition to our impedance problem.
Therefore, we can claim that $u_n(\mathbf x)$ converges strongly in $L^2(\Omega)$ to $u(\mathbf x)$.

The proof is complete.
\end{proof}

\begin{rem}\label{rem:mosco}
In the following, we discuss the assumption of Proposition~\ref{prop mosco}   that $H^1(\Omega\setminus K_n)$ converges to $H^1(\Omega\setminus K)$ in the sense of Mosco. 
The relationship between the Mosco convergence of a general Sobolev $W^{k,p}(K_n)\rightarrow W^{k,p}(K)$ space and the convergence of its defined bounded domain $K_n\rightarrow K$ in the sense of Hausdorff distance can be founded in \cite[Proposition~2.11]{fornoni2023mosco}. 
According to \cite[Proposition 2.11]{fornoni2023mosco}, one can state that the following results also hold when $K\in\mathcal{B}$ is an admissible obstacle with the impedance boundary condition.
Denote $A_n = H^1(\Omega\setminus K_n), n\in\mathbb{N}$, and $A=H^1(\Omega\setminus K) $.
Then we have the following results:
\begin{enumerate}
			\item if $K_n\rightarrow K$ in the Hausdorff complementary topology and $|K_n\setminus K| \rightarrow 0$ as $n \rightarrow \infty$, then condition (2) of Mosco convergence in Definition~\ref{def:mosco in sobolev space} holds;
			\item if $K_n\subset K$ for every $n\in\mathbb{N}$ and $|K\setminus K_n| \rightarrow 0$ as $n\rightarrow\infty$, then condition (1) of Mosco convergence in Definition~\ref{def:mosco in sobolev space} holds;
			\item if $K_n\rightarrow K$ in the Hausdorff complementary topology and $K_n\subset K$ for every $n\in\mathbb{N}$, then $A_n$ converges to $A$ as $n\rightarrow\infty$ in the sense of Mosco; 
\end{enumerate}
where $|K|$ denotes the Lebesgue measure of $K$.
\end{rem}

%----------------------------------------------------------------------
\subsection{Uniformly bounded estimations}

Using Mosco convergence, we can derive a uniformly bounded estimate for the corresponding solution of the equations for obstacles in the admissible class $\mathcal{B}$.

\begin{prop}\label{uniform L2 bound}
For fixed constant $R>0$, fixed wave number $k>0$ and propagation direction $\mathbf p\in \mathbb S^{1}$ of $u^i(\mathbf x)=e^{\mathrm i k\mathbf x\cdot \mathbf p}$, let the admissible class $\mathcal{B}$ be as in Definition~\ref{def:Class B}.
 For any $K \in \mathcal B$, if $u(\mathbf x)$ is the total wave field of \eqref{impedance problem} associated with $K$ and the incident wave $u^i$, then
\begin{equation}\label{eq:uniform bound}
		 \| u\|_{L^{2}(B_{R+1}\setminus K)} \leq \mathcal{E},	
\end{equation}
where $\mathcal E$ is a positive constant depending on $R$, $k$, $\mathbf p$ only.

Furthermore, for any $K \in \mathcal B$, one has $u(\mathbf x)\in C^\alpha(B_{R}\setminus \overline K)$. The following bounds 
	\begin{equation}\label{eq:6.1}
		\|u(\mathbf x) \|_{C^\alpha(B_{R}\setminus \overline K) }\leq \mathcal E_R,
	\end{equation}
	and
	\begin{equation}\label{eq:6.2}
		\|u(\mathbf x) \|_{H^1(\partial K)}\leq \mathcal E_H,
	\end{equation}
	hold, where the constants $\mathcal E_R,\ \mathcal E_H$ depending on $R$, $k$, $\mathbf p$ only.
%	And more importantly, we can demonstrate that the two upper bound estimations above can be uniformly established in the class $\mathcal B$.
\end{prop}

\begin{proof}
We first prove \eqref{eq:uniform bound} by contradiction. 
Hence, there exists a sequence $\{K_n\}_{n=1}^\infty$ such that $K_n\in \mathcal B$ and 
$$
\Vert u_n \Vert_{L^2(B_{R+1}\setminus K_n)}=a_n \geq n, \quad \forall n\in \mathbb N, 
$$ 
where $u_n(\mathbf x)$ is the total wave field to \eqref{impedance problem} associated with $(K_n,\eta_n)$. 
According to Proposition~\ref{prop:compact}, the admissible obstacle class $\mathcal{B}$ is compact with respect to the Hausdorff distance. 
Without loss of generality, we assume that $\{K_n\}_{n\in\mathbb{N}}\in\mathcal B$ converges to $K\in \mathcal B$ with respect to Hausdorff distance.
Similarly, we assume that $\{\eta_n\}_{n\in\mathbb{N}}\in\Xi_{\mathcal B}$  converges to $\eta\in\Xi_{\mathcal B}$ strongly in the sense of $L^2$.
    
Let $v_n(\mathbf x) = u_n(\mathbf x)/a_n$ and  $v^s_n(\mathbf x)=u^s_n(\mathbf x)/a_n$, where $u^s_n$ is the scattered wave field to \eqref{impedance problem} associated with $(K_n,\eta_n)$. 
Obviously, one has that $\Vert v_n \Vert_{L^2(B_{R+1}\setminus K_n)} =1 $.
Therefore, up to a subsequence, we can assume that $v_n(\mathbf x)$ converges to a function $v(\mathbf x)$ strongly in $L^2$-norm, where  $v(\mathbf x)$ satisfies $\Vert v \Vert_{L^2(B_{R+1}\setminus K)} =1 $, and
\begin{equation*}
	\begin{cases}
		&{\Delta v+k^2v=0  \hspace*{3.1cm} \mbox{in}\quad \mathbb R^2\setminus K},\\[8pt]
		&\frac{\partial v}{\partial \nu}+\eta v=0   \hspace*{3.4cm} \mbox{on}\quad \partial K. 
	\end{cases}
\end{equation*} 
Due to $v_n(\mathbf x)=u^i(\mathbf x)/a_n+u_n^s(\mathbf x)/a_n$, where $u^i(\mathbf x)$ is the incident wave defined in \eqref{eq incident wave}, it is clear that $\{\Vert u^s_n/a_n \Vert_{L^2(B_{R+1}\setminus K)}\}_{n\in\mathbb{N}}$ is uniformly bounded.
By virtue of Proposition~\ref{prop mosco}, we know that $u^s_n(\mathbf x)/a_n$, up to a subsequence, converges to a function $w(\mathbf x)$ strongly in $L^2$-norm, where $w(\mathbf x)$ satisfies
\begin{equation*}
	\begin{cases}
		&{\Delta w+k^2w=0  \hspace*{3.1cm} \mbox{in}\quad \mathbb R^2\setminus K},\\[8pt]
		&{\lim_{r\to \infty}r^{1/2}(\partial _rw-\mathrm i kw)=0\hspace*{0.9cm} r=\vert \mathbf x\vert}.
	\end{cases}
\end{equation*}
Since $\displaystyle \lim_{n\rightarrow \infty}a_n=+\infty$, we know that $v(\mathbf x)=w(\mathbf x)$.
Therefore  $v(\mathbf x)$ satisfies
\begin{equation*}
	\begin{cases}
		&{\Delta v+k^2v=0  \hspace*{3.1cm} \mbox{in}\quad \mathbb R^2\setminus K},\\[5pt]
		&\frac{\partial v}{\partial \nu}+\eta v=0   \hspace*{3.4cm} \mbox{on}\quad \partial K,  \\[5pt]
		&{\lim_{r\to \infty}r^{1/2}(\partial _rv-\mathrm i kv)=0\hspace*{1cm} r=\vert \mathbf x\vert}.
	\end{cases}
\end{equation*}

By the uniqueness of the exterior impedance boundary value problem (cf. \cite[Theorem 3.16]{colton2019inverse}), it yields that $v(\mathbf x)=0$, where we get the contradiction. 

 \medskip
	In the following we want to prove \eqref{eq:6.1} and \eqref{eq:6.2}. According to the regularity result for the impedance scattering problem \eqref{impedance problem} discussed in \cite{clm2012} and \cite[Theorem 3.2, Theorem 3.4]{alessandrini2013stable}, where the underlying obstacle $K$ is a polygon or has a Lipchitz boundary, the corresponding total wave field $u$ has the regularity  $u \in C^{\alpha} ({B_R \backslash\overline K })$ ($
	0<\alpha<1$). Using \cite[Theorem 3.4]{alessandrini2013stable}, it yields that $u\in H^1(\partial K)$. Similar to \eqref{eq:uniform bound}, we can prove \eqref{eq:6.1} and  \eqref{eq:6.2} by a contradict argument. The detailed proofs are omitted. 
	
The proof is complete.
\end{proof}	

%For any $K\in\mathcal B$, the estimation \eqref{eq:6.1} is a more or less standard regularity estimate up to the boundary. We just point out that the estimates \eqref{eq:6.1} and \eqref{eq:6.2} can be obtained under the impedance direct sacttering problem and only require the Lipchitz regularity of the $\partial K$, for details (cf.\cite[Theorem 3.2, Theorem 3.4]{alessandrini2013stable}) respectively. And more details about the standard regularity estimation for impedance scattering problem can be found in \cite{cakoni_direct_2001}.
		
	%	To demonstrate that the upper bound estimation can be based on the Mosco convergence of the solution space holds uniformly in the admissible class $\mathcal B$, we use a contrary argument similar to the proof of \eqref{eq:uniform bound}. Therefore, we just omit the details.

%It is not difficult to prove a uniform lower bound for the admissible obstacle class.
%Similarly, we assume that there is a sequence $\{K_n\}_{n\in\mathbb{N}}\in\mathcal B$ converges to $K\in \mathcal B$ with respect to Hausdorff distance satisfying
%$$
%\Vert u_n \Vert_{L^2(B_{R+1}\setminus K_n)}=a_n, \quad \forall n\in \mathbb N,\ \text{and}\ \lim_{n\rightarrow \infty}a_n=0.
%$$
%We immediately find that this contradicts the result of Rellich's lemma in \cite[Lemma 2.12]{colton2019inverse}.

Accordingly, we can obtain the following result immediately by \cite[Theorem~2.8]{adams1975sobolev}. 

\begin{lem}\label{uniform infity bound}
Let us fix $K\in \mathcal{B}$, $u(\mathbf x)\in L^{2}(B_{R+1} \setminus K)$ and $\Vert u \Vert_{L^2(B_{R+1}\setminus K)} \leq \mathcal E $, where $u$ is the solution to \eqref{impedance problem} associated with $K$.  
Then $u(\mathbf x)\in L^{\infty}(B_{R+1}\setminus K)$, and $\Vert u \Vert_{L^{\infty}(B_{R+1}\setminus K)} \leq \mathcal E$.
\end{lem}

By establishing uniform boundedness estimates for admissible classes $\mathcal B$,  we establish the necessary conditions for the following argument.
We introduce the following three-sphere inequalities which can be found, for instance, in \cite{brummelhuis1995three}.

\begin{lem}\label{three ball th}\cite[Lemma 3.5]{rondi2008stable} 
There exist positive constants $\tilde{R},\ C$ and $c_1,\ 0<c_1<1$, depending on $k$ only, such that for every $0<r_1<r_2<r_3\leq \tilde{R}$ and any function $u(\mathbf x)$ such that $$\Delta u+k^2u=0 \ in \ B_{r_3}.$$
For any $r,r_2<s<r_3$,  we have
\begin{equation}
		{\left\| u \right\|_{{L^\infty }({B_{{r_2}}})}} \le {{C(}}1 - ({r_2}/s){)^{ - 3/2}}\left\| u \right\|_{{L^\infty }({B_{{r_3}}})}^{1 - \beta }\left\| u \right\|_{{L^\infty }({B_{{r_1}}})}^\beta,
\end{equation}
for some $\beta$ such that
\begin{equation*}
	{{{c_1}(\ln ({r_3}/s))} \over {\ln ({r_3}/{r_1}))}} \le \beta  \le 1 - {{{c_1}(\ln (s/{r_1}))} \over {\ln ({r_3}/{r_1}))}}.
\end{equation*}
\end{lem}

%----------------------------------------------------------------------
\section{Propagation of measurement errors:  from far-field to boundary}\label{sec:near field}

In the following, we discuss how the smallness from the far-field pattern propagates to the near-field and then use the three-sphere inequality as well as the local H\"older continuity to show how the smallness from the near-field pattern propagates to the boundary. 
Those error propagation are key ingredients in the stability proof of Theorem~\ref{mian result}.

%---------------------------------------------------------------------
\subsection{Stability Estimates: from far-field to near-field}

In this subsection, we aim to show that if the error between $u^{\infty}_{(K,\eta)}(\hat{\mathbf x})$ and $u^{\infty}_{(K',\eta')}(\hat{\mathbf x})$ are small, then the error between $u(\mathbf x)$ and $u'(\mathbf x)$ are also small in a near-field $B_{2R}\setminus B_{R}$, where $K\Subset B_R$ and $K'\Subset B_R$. 
Here $B_R$ is a disk centered at the origin with a radius $R\in \mathbb R_+$.

Throughout the rest of this section, let
\begin{align}\label{eq:w}
	w(\mathbf x)=u(\mathbf x) - u'(\mathbf x),
\end{align}
where $u(\mathbf x)$ and $u'(\mathbf x)$ are total wave field to \eqref{impedance problem} associated with the impedance obstacle obstacles $K$ and $K'$, respectively. 
Then $(\Delta +k^2)w=0 $ in $B_{R}\setminus (K\cup K')$.

Utilizing the corresponding results in \cite{isakov1992stability} and \cite{rondi2015stable}, we have the following lemma.
\begin{lem}\label{lem near error}
Let $K$ and $K'$ be two admissible obstacles described in Definition~\ref{def:Class B}, where $K\Subset B_R$ and $K'\Subset B_R$.	 
Assume that $u^\infty_{(K,\eta')}(\hat{\mathbf x})$ and ${{{u^{\infty }}}_{(K',\eta')}(\hat{\mathbf x})}$ are the far-field patterns of the scattered waves $u^s(\mathbf x)$ and $(u^{s})'(\mathbf x)$ associated with the impedance obstacle obstacles $K$ and $K'$, respectively.  
If the far-field error between $u^{\infty}_{(K,\eta)}(\hat{\mathbf x})$ and $u^{\infty}_{(K',\eta')}(\hat{\mathbf x})$ satisfies \eqref{eq:far-field error}, then there exist $\zeta \in \mathbb R_+$ and $\mathbf{x_0}\in B_{2R}\setminus B_R$ such that $R+1+\zeta \leq \Vert \mathbf{x_0} \Vert \leq 2R$ and
\begin{equation}\label{eq: near error}
		\left\|w(\mathbf x) \right\|_{L^{\infty} (B_{\Vert \mathbf{x_0} \Vert + \zeta} \setminus{\overline{B _{\Vert \mathbf{x_0} \Vert - \zeta}}})}  \leq \varepsilon _1:=\exp(-C_a(-\ln \varepsilon)^{1/2}),
\end{equation}
where the constant $C_a$ depending the a-priori parameters only and $w(\mathbf x)$ is defined in \eqref{eq:w}. 
For any $\mathbf {x_1}\in B_{\Vert \mathbf{x_0} \Vert + \zeta} \setminus{\overline{B _{\Vert \mathbf{x_0} \Vert - \zeta}}}$, it yields that
\begin{equation}\label{eq:x1}
	\left\|w(\mathbf x)  \right\|_{L^{\infty} (B_{\zeta}(\mathbf x_1))} \leq \varepsilon_1 , \ \ \   \forall \mathbf {x_1}\in B_{\Vert \mathbf{x_0} \Vert + \zeta} \setminus{\overline{B _{\Vert \mathbf{x_0} \Vert - \zeta}}}.
\end{equation}
\end{lem}

%-------------------------------------------------------------------
\subsection{Stability Estimates: from near-field to boundary}
The main purpose of this subsection is to estimate $\left|u(\mathbf x)-u'(\mathbf x)\right|$ and $\left|\nabla[u(\mathbf x)-u'(\mathbf x)]\right|$ in Proposition~\ref{Lem:Boundary estimation}. 
We use the three-sphere inequality iteratively to propagate the near-field data to the near boundary. 
Then, we use the local H\"older continuous to propagate the near-field data to the boundary of the obstacle.  
 
Without loss of generality, up to swapping $K$ and $K'$, we can find a point $\mathbf{y_2} \in K$ such that $\mathfrak{h}=d_M(K, K')=\dist(\mathbf{y_2}, K')$, where $\mathbf{y_2}$ is a vertex of $K$ (cf.\cite[Lemma~2]{gregoire1998hausdorff}).

In Proposition~\ref{prop disks chain}, we construct a sequence of disks to propagate the near-field data to the neighborhood of the vertex $\mathbf{y_2}$ of $K$ by using three-sphere inequality iteratively, where the disks we construct can not touch another obstacle $K'$. 
Figure~\ref{figure1} shows a schematics illustration of this constructed procedure.
Before that we first give the regularity result on $w(\mathbf x)$ near the vertex $\mathbf y_2$. 
Throughout the rest of this section, denote that $Q_h=B_h(\mathbf{y_2}) \cap K$, $P_h=B_h(\mathbf{y_2}) \setminus K$ and  $\Gamma^\pm_{h}=B_h(\mathbf{y_2}) \cap \partial K$, where $h\in (0,1)$ will be specified later.

Using a similar argument in \cite[Proposition~5.5]{alessandrini_stability_2009}, we can obtain the following Proposition.

\begin{prop}\label{prop Gr}
Let $K$, $K'\in \mathcal{B}$ be the admissible polygon obstacles.
Suppose that $G$ is a  bounded, connected, non-empty open and Lipschitz subset of $\Omega\setminus(K\cup K')$ in the context of this paper.
We define a subset $G_r$ of $G$ for $r \in \mathbb R_{+}$ as follows
\begin{equation}\label{eq Gr}
		G_r := \{\mathbf x \in G| \dist(\mathbf x, \partial G) > r\}.
\end{equation}
We assumed that the length of $\partial G$, denoted by $|\partial G|$, is finite.
Then, there must exist a constant $r_m \in \mathbb R_{+}$ only depending on the a-prior information of admissible obstacle class $\mathcal{B}$, such that $G_r$ is nonempty and connected for $r\in (0, r_m)$.
\end{prop}

\begin{figure}[htbp]
	\centering
\includegraphics[height=8cm,width=8cm]{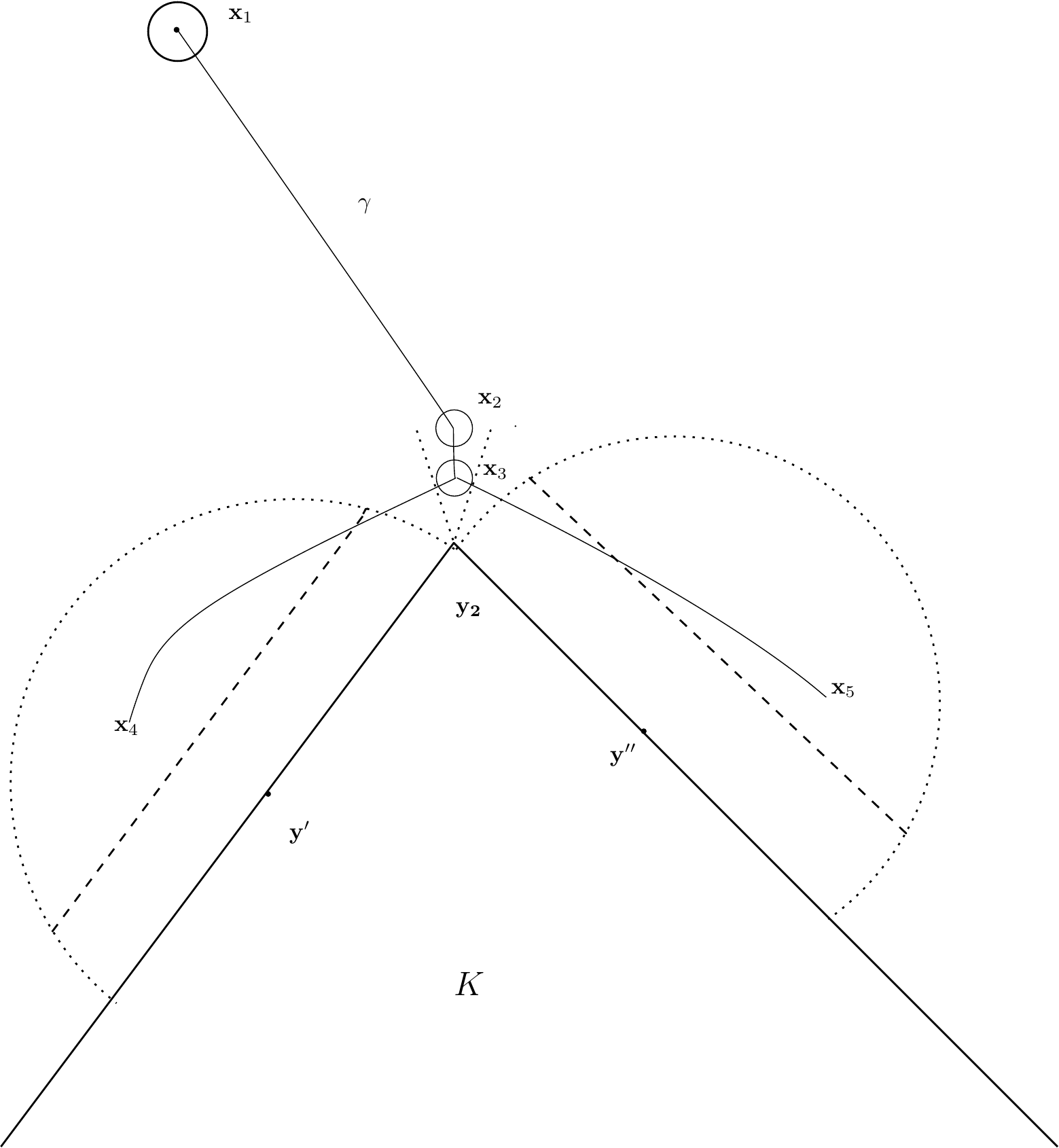}
	\caption{Schematic illustration of propagations from near-field to the boundary.}
	\label{figure1}
\end{figure}

\begin{defn}\label{Def:chain of ball}
Let $M\in \mathbb N$ and $r>0$.  
A sequence of disks $B_1,...,B_M$ is said to be a \textit{chain of disks with radius $r$}, if it satisfies the follow conditions:
\begin{itemize}
		\item [(1)] $r<r_m/5$;
		\item [(2)] the radius of each $B_i$ is $r$;
		\item [(3)] the center-to-center distance of $B_i$ to $B_{i+1}$ is at most $r$.
\end{itemize}
\end{defn}

\begin{lem}\label{Lem:local Holder continuous}
 Let $K$ and $K'$ be admissible obstacles described in Definition~\ref{def:Class B} and $K, K'\subset B_R=\Omega$.  
 Let $w(\mathbf x)\in H^{1}(\Omega\setminus K)$ be defined in \eqref{eq:w}. 
 If $0<h<1$ satisfies 
 \begin{align}\label{eq:cond h}
 		Q_h\cap K'=\emptyset,
 \end{align}
 then there exist some constants $0<h_2<h$, $0<\alpha<1$, and two point $\mathbf{y'}\in \Gamma^+_{h}$ and $\mathbf{y''} \in \Gamma^-_{h}$ such that $w(\mathbf x)\in C^{1,\alpha}(B^{+}_{h_2}(\mathbf{y'}))$ and $w(\mathbf x)\in C^{1,\alpha}(B^{+}_{h_2}(\mathbf{y''}))$, where $B^{+}_{h_2}(\mathbf{y'}) \Subset \Omega\backslash K' $ and $B^{+}_{h_2}(\mathbf{y''})\Subset \Omega\backslash K'$ are two small semi-circular domain of disks $B_{h_2}(\mathbf{y'})$ and $B_{h_2}(\mathbf{y''})$ respectively. 
 Furthermore, it yields that
 \begin{equation}\label{eq:H^3}
 		\mathcal T=\max\{{\left\| w \right\|_{{C^{1,\alpha }}({B^ +_{h_2 }}(\mathbf{y'}))}},{\left\| w \right\|_{{C^{1,\alpha }}({B^ +_{h_2 }}(\mathbf{y''}))}} \} \le C{\left\| w \right\|_{{H^1}(\Omega \backslash (K\cup K'))}},
 \end{equation}
 where the constant $C>0$ depends on the a-priori parameters only.
\end{lem}

\begin{proof}
Given $h\in (0,1)$ satisfying $Q_h\cap K'=\emptyset$, we choose $h_2=h/2$. 
Let $\mathbf y'$ and $\mathbf y''$ be the middle points of $\Gamma_h^+$ and $\Gamma_h^-$, respectively. 
Hence $B^{+}_{h_2}(\mathbf{y'}) \Subset \Omega\backslash (K\cup K') $ and $B^{+}_{h_2}(\mathbf{y''})\Subset \Omega\backslash(K\cup K')$. 
Since $w(\mathbf x)\in H^{1}(\Omega\setminus K)$, by interior regularity estimation of an elliptic partial differential operator, we know that $u$ is analytic in $B^{+}_{h_2}(\mathbf{y'}) \cup B^{+}_{h_2}(\mathbf{y''}) $ and \eqref{eq:H^3} holds.
\end{proof}

By Lemma~\ref{three ball th} and using a similar argument in \cite[Lemma 3.5]{rondi2008stable}, we can obtain the following lemma.

\begin{lem}\label{Lem:Three ball theorem iteration}
Let $U\subset \Omega$ be a bounded connected domain, and $ \gamma \subset  U$ be a rectifiable curve which links two distinct points $\mathbf x,\mathbf y \in U$ such that $B_{4r}(\gamma) \subset U$. 
Let $w(\mathbf x) \in L^{\infty}(U)$ satisfy the Helmholtz equation, and $\Vert w \Vert_{L^{\infty}(U)}  \leq \mathcal{E}$.
If we can construct a chain of disks with radius $r$ connecting $\mathbf x$ and $\mathbf y$ along the curve $\gamma$, then we have
\begin{equation}
	{\left\| w \right\|_{{L^{\infty} }({B_r}(\mathbf y))}} \le {C_s}{{\mathcal E}_2}\left\| w \right\|_{{L^{\infty} }({B_r}(\mathbf x))}^{\beta^{d_\gamma /r + 1}} \le {C_s}{{\mathcal E}}\left\| w \right\|_{{L^{\infty} }({B_r}(\mathbf x))}^{\beta^{d_\gamma /r + 1}},
\end{equation}
where the constants $\mathcal{E}$ and $\beta$ are given in  Lemma~\ref{uniform infity bound} and Lemma~\ref{three ball th} respectively, the  constant $C_s>0$ depends on  a-priori parameters only. 
Here $d_\gamma$ is the length of $\gamma$.
\end{lem}

\begin{prop}\label{prop disks chain}
For a fixed point $\mathbf x_1$ given in \eqref{eq:x1}, there exits two finite chain disks $\cup_{j=1}^{M_1} B_r(\mathbf x_j)$ and $\cup_{j=1}^{M_2} B_r(\mathbf z_j)$  with radius $r$ with the starting disk centered at $\mathbf z_1=\mathbf x_1$ such that $\{\mathbf x_j\}_{j=1}^{M_1} \in \gamma_1$ and $\{\mathbf z_j\}_{j=1}^{M_2} \in \gamma_2$, where $\gamma_1$ and $\gamma_2$  are two rectifiable curves.
Furthermore it yields that $\left(\cup_{j=1}^{M_1} B_{4r}(\mathbf x_j)\right) \cup \left(\cup_{j=1}^{M_2} B_{4r}(\mathbf z_j)\right)\Subset \mathbb R^2 \backslash (\overline{ K\cup K'})$,  $B_{4r}(\mathbf x_{M_1}) \Subset B^{+}_{h_2}(\mathbf{y'})$ and $B_{4r}(\mathbf z_{M_2}) \Subset B^{+}_{h_2}(\mathbf{y''})$, where $B^{+}_{h_2}(\mathbf{y'})$ and $B^{+}_{h_2}(\mathbf{y''})$ are defined in Lemma~\ref{Lem:local Holder continuous}.
Here $r>0$ depends on a-prior information only.
\end{prop}

\begin{proof}
Let $\mathbf{y_2} \in \mathcal{V}(K)$, where $\mathfrak{h}=d_M(K, K')=\dist(\mathbf{y_2},K')$. Firstly, we construct open cones at the vertex $\mathbf{y_2}$, such that the disk for the propagation of smallness from the near-field to the boundary of the obstacle can be tangent to the open cones when passing near the vertex.
We use the properties of the open cone to guarantee that the conditions of the Lemma~\ref{Lem:Three ball theorem iteration} hold.
According to the uniform exterior cone condition of the admissible class $\mathcal{B}$, there exists a direction $\mathbf{\omega_1}\in \mathbb{S}^{1}$ such that the open cone $\mathcal{C}(\mathbf{y_2},\mathbf{\omega_1},\delta_1,\theta_1)$ is contained in $\Omega \backslash (K\cup K')$.
The chosen direction of the cone provides a guarantee that another obstacle $K'$ is avoided when the smallness is propagated by the three-sphere inequality in Lemma~\ref{Lem:Three ball theorem iteration}. 

We can find some point $\mathbf{x_3} \in\mathcal{C}(\mathbf{y_2},\mathbf{\omega_1},\delta_1,\theta_1) $ such that $\mathbf{x_2}=\mathbf{x_3}+t_2\mathbf{\omega_1}$ in the bisecting line of the cone $\mathcal{C}(\mathbf{y_2},\mathbf{\omega_1},\delta_1,\theta_1)$, where $t_2$ belongs to a finite interval in $\mathbb R_+$.
By adjusting the position of $\mathbf{x_3}$, we can choose $r\in \mathbb R_+$ such that the open cone $\mathcal{C}(\mathbf{y_2},\mathbf{\omega_1},\delta_1,\theta_1) $ is the tangential cone of disk $B_{4r}(\mathbf{x_3})$.

By Proposition~\ref{prop Gr}, the positive constant $r_m$ can be found only depending on a-prior information of $K, K'$ such that the set $G_{4r}$ defined as in \eqref{eq Gr} is connected for $0<4r<r_m$.
According to assumption (1) in Definition~\ref{Def:chain of ball}, we can find a chain of disks that satisfies the above conditions.
Therefore, it is not difficult to find two points $\mathbf{x_4}\in G_{4r}\cap B^+_{h_2}(\mathbf{y'}),\mathbf{x_5}\in G_{4r}\cap B^+_{h_2}(\mathbf{y''})$  satisfying  $B_{4r}(\mathbf x_4) \Subset B^{+}_{h_2}(\mathbf{y'})$ and $B_{4r}(\mathbf x_5) \Subset B^{+}_{h_2}(\mathbf{y''})$ respectively if $r<h_2/8$ is small enough.

Let a fixed point $\mathbf x_1$ fulfill \eqref{eq:x1} and a rectifiable curve $\gamma_0$ passes $\mathbf{x_1}, \mathbf{x_2}$ and $\mathbf{x_3}$.
Hence we can construct a chain of disks with radius $r$, joining the points
$\mathbf{x_1}, \mathbf{x_2},\mathbf{x_3}, \mathbf{x_4}$ made of disks centered on a rectifiable curve $\gamma_1$.
The first disk of this chain is $B_r(\mathbf{x_1})$, and the last disk is $B_r(\mathbf{x_4})$.
Similarly, we construct the chain of disks that connect $\mathbf{x_1}, \mathbf{x_2},\mathbf{x_3},\mathbf{x_5}$ in sequence with the center of the disk on $\gamma_2$.
The first disk of this chain is $B_r(\mathbf{x_1})$, and the last disk is $B_r(\mathbf{x_5})$.

By the construction of those chains of disks, we know that both propagation smallness from $\mathbf{x_1}$ to $\mathbf{x_4}$ and $\mathbf{x_1}$ to $\mathbf{x_5}$ satisfy the conditions of Lemma~\ref{Lem:Three ball theorem iteration}.
Without loss of generality, we rename all the center points of disks on the $\gamma_1$ and $\gamma_2$ as $\{\mathbf x_j\}_{j=1}^{M_1}$ and $\{\mathbf z_j\}_{j=1}^{M_2}$ respectively, such that $\gamma_1$ and $\gamma_2$ are two rectifiable curves.

The proof is complete.
\end{proof}

\begin{prop}\label{Lem:Near field to boundary}
Let $K$ and $K'$ be admissible obstacles described in Definition~\ref{def:Class B}. 
Recall that positive parameters $\beta\in (0,1)$, $\zeta$, $d_\gamma $, $0<h_2<h<1$, $r_m$, $\mathcal{T}$, and $\alpha \in (0,1)$ are defined in Definition~\ref{Def:chain of ball}, Lemmas~\ref{lem near error}, \ref{Lem:Three ball theorem iteration} and \ref{Lem:local Holder continuous} respectively. 
Suppose that $w(\mathbf x)\in L^{\infty}(B_{2R})$ be defined in \eqref{eq:w}. 
If \eqref{eq:x1} is fulfilled and
\begin{equation}\label{eq upper boudnd near error}
		{\varepsilon _1} \le {[\exp \exp ({{4|d_{\gamma} ||\ln {\beta}|} \over {(1 - \alpha )\min (r_m,h/4,\zeta )}})]^{ - 1}},
\end{equation}
where $\varepsilon _1$ is given in \eqref{eq:x1}, then there exist a constant $C_f>0$, such that
$$
	\sup_{\mathbf x\in\Gamma^{\pm}_h}|w(\mathbf x)| \leq  C_f(\ln|\ln \varepsilon_1|)^{-\alpha}, 
$$
	where $C_f>0$ depends on the a-priori parameters only.
\end{prop}

\begin{proof}
Let us fix a point $\mathbf{x_1}$ in the near-field satisfying \eqref{eq:x1}. 
In addition, we have constructed two chains of disks $\left(\cup_{j=1}^{M_1} B_r(\mathbf x_j)\right) \cup \left(\cup_{j=1}^{M_2} B_r(\mathbf z_j)\right)\Subset \mathbb R^2 \backslash (\overline{ K\cup K'})$ to propagate smallness from near-field to the neighborhood of the obstacle.
Next, we would like to introduce the argument about propagation smallness to the boundary of the obstacle.

For $B^{+}_{h_2}(\mathbf{y'})\cup B^{+}_{h_2}(\mathbf{y''})$, without loss of generality, we only give the analysis for $B^+_{h_2}(\mathbf{y'})$ in the following. 
It is convenient to divide the $B^+_{h_2}(\mathbf{y'})$ into the following two sets
$$
\Sigma_1 := \{\mathbf{x} \in B^{+}_{h_2}(\mathbf{y'})| \dist(\mathbf x, \partial \Gamma^+_{h}) \geq 4r\},\quad \Sigma_2 := \{\mathbf x \in B^{+}_{h_2}(\mathbf{y'})| \dist(\mathbf x, \partial \Gamma^+_{h}) \leq 4r\}.
$$
    
Using the geometric structure of the chain of disks constructed by Proposition~\ref{prop disks chain}, we know that $\mathbf{x}_{M_1}\in \Sigma_1$ and the chain $\gamma_1$ propagating smallness from near-field to $B_r(\mathbf{x}_{M_1})$. 
Similarly, one has $\mathbf{z}_{M_2}\in \Sigma_2 $ and curve $\gamma_2$ propagating smallness from near-field to $B_r(\mathbf{z}_{M_w})$.
Let $d_{\gamma_1}$ and $d_{\gamma_2}$ be lengths of curves $\gamma_1$ and $\gamma_2$ respectively.
Denote $d_\gamma=\max\{d_{\gamma_1},d_{\gamma_2}\}$.
Let us choose
\begin{equation}\label{eq r}
    	r=r(\varepsilon_1)=\frac{ d_\gamma|\ln \beta|}{(1-\alpha)\ln |\ln\varepsilon_1|},
\end{equation}
where $\beta$ is given by Lemma~\ref{three ball th} and $r>0$.
The upper bound on $\varepsilon_1$ implies that the condition of Proposition~\ref{prop disks chain} is satisfied. 
Using Lemma~\ref{Lem:Three ball theorem iteration}, we consider two semicircular domains respectively, where
$$
   \max\{\left|w(\mathbf x_{M_1})\right|,\left|w(\mathbf z_{M_2})\right|\} \le {C_s}{{\mathcal E}}{\varepsilon_1}^{\beta^{d\gamma /r + 1}}.
$$

For any $\mathbf x \in \Gamma^+_h$, we can find $\mathbf {x'}$ in a subset of $\Sigma_2$ such that $|\mathbf x-\mathbf {x'}| = 4r$.
Therefor we assume the $\mathbf{x'}\in B_{4r}(\mathbf{x}_{M_1})$ satisfying $B_{4r}(\mathbf x_{M_1}) \Subset B^{+}_{h_2}(\mathbf{y'})$.
The	upper bound on $\varepsilon_1$ implies  $4r < \min (r_m,h/4,\zeta)$. Thus $|\mathbf x-\mathbf{x}_{M_1}| \leq |\mathbf x-\mathbf{x'}|+|\mathbf{x'}-\mathbf{x}_{M_1}| \leq 8r$.
By the local H\"older continuity of $w(\mathbf x)$ in Lemma~\ref{Lem:local Holder continuous}, we have
\begin{equation}
\begin{aligned}
	|w(\mathbf{x})|&\leq \mathcal{T}|\mathbf{x}-\mathbf{x}_{M_1}|+|w(\mathbf{x}_{M_1})|\\
	&\leq \mathcal{T}8^{\alpha}r^{\alpha}+C_s\mathcal{E} {\varepsilon_1}^{\beta^{ d_\gamma/r+1}}\\
	&\leq (8^{\alpha}r^{\alpha}+{C\varepsilon_1}^{c_2^{ d_\gamma/r+1}})\mathcal{T}_1,
\end{aligned}
\end{equation}
where $\mathcal E$ comes from Lemma~\ref{Lem:Three ball theorem iteration} and $\mathcal T_1=\max\{\mathcal T,\mathcal E\}$.

According to \eqref{eq r}, it implies that
\begin{equation*}
	\frac{ d_\gamma}{r}=\frac{(1-\alpha)\ln |\ln\varepsilon_1|}{|\ln \beta|}=\log_{\beta}(|\ln\varepsilon_1| ^{\alpha-1}),
\end{equation*}
and hence
\begin{equation*}
\begin{aligned}
  {\varepsilon_1}^{\beta^{\mathrm d_\gamma/r+1}}
  &={e^{ - |\ln {\varepsilon _1}|\beta^{\mathrm {d_\gamma}/r + 1}}}\\
  &={e^{ - |\ln {\varepsilon _1}|\beta\log _{{\beta}}^{}(|\ln{\varepsilon _1}{|^{\alpha  - 1}}) + 1}}\leq {e^{ - {\beta}|\ln {\varepsilon _1}{|^\alpha }}}\leq \frac{1}{\beta|\ln {\varepsilon _1}{|^\alpha }}\leq \frac{1}{\beta}(\ln|\ln \varepsilon_1|)^{-\alpha}.
\end{aligned}
\end{equation*}
Therefore we have a constant $C_f>0$, such that
\begin{equation*}
\begin{aligned}
	|w(\mathbf x)|
	& \leq \frac{(16R|\ln \beta|/(1+\alpha))^{\alpha}+C/\beta}{(\ln|\ln \varepsilon_1|)^{\alpha}}\ \mathcal{T}_1\\
	&\leq C_f(\ln|\ln \varepsilon_1|)^{-\alpha},\\
\end{aligned}
\end{equation*}
where the positive number $C_f$ depends on the a-priori parameters only.

The proof is complete.
\end{proof}

\begin{prop}\label{Lem:Boundary estimation}
Let $K$, $K'\in \mathcal{B}$ be the admissible polygon obstacles. 
Let $u^i(\mathbf x)$ be the fixed incident plane wave, and $u(\mathbf x), u'(\mathbf x)$ be total waves satisfying the impedance system \eqref{impedance problem}, and their far-field pattern are $u^{\infty}_{(K,\eta)}(\hat{\mathbf{x}})$ and $u^{\infty}_{(K',\eta')}(\hat{\mathbf{x}})$ respectively.

If the condition \eqref{eq:far-field error} is fulfilled, then
\begin{equation}\label{eq sup b}
    	\sup_{\partial \Gamma^\pm_{h}}(|u-u'|+|\nabla(u-u')|)\leq C_b(\ln|\ln 1/\varepsilon|)^{-\frac{1}{2}}.
\end{equation}
where the positive constant $C_b$ depends on the a-priori parameters only.
\end{prop}

\begin{proof}
By Lemma~\ref{uniform infity bound}, we establish a uniform bounded estimate and $\Vert w\Vert_{L^\infty(\Omega\setminus K)}\leq \mathcal E$.
Recall that $w(\mathbf x)=u(\mathbf x)-u'(\mathbf x) = u^s(\mathbf x)-u'^s(\mathbf x)$ and $\varepsilon$ denote the far-field.
Then we can get the near-field error $\varepsilon_1$ in a given circular domain by Lemma~\ref{lem near error} in $B_{\Vert \mathbf{x_0} \Vert + \zeta} \setminus{\overline{B _{\Vert \mathbf{x_0} \Vert - \zeta}}}$, and
\begin{equation}\label{eq f to n}
		\left\|u - u' \right\|_{L^{\infty} (B_{\Vert \mathbf{x_0} \Vert + \zeta} \setminus{\overline{B _{\Vert \mathbf{x_0} \Vert - \zeta}}})} = \varepsilon_1 \leq \exp(-C_a(-\ln \varepsilon)^{1/2}).
\end{equation}
If $\varepsilon$ is sufficiently small, there exists a constant $\varepsilon_m$ such that $\varepsilon_1$ satisfies its upper bound \eqref{eq upper boudnd near error}. 

The next step is to use the propagation of smallness by Proposition ~\ref{Lem:Near field to boundary} for $w(\mathbf x) = u(\mathbf x)-u'(\mathbf x)$ and  $ \nabla w(\mathbf x) = \nabla u(\mathbf x)-\nabla u'(\mathbf x)$.
By Lemma~\ref{Lem:local Holder continuous} we have $$\mathcal T \le C{\left\| u \right\|_{{H^1}(\Omega \backslash K)}}.$$
Therefore, both $w(\mathbf x)=u(\mathbf x)-u'(\mathbf x)$ and $\nabla w(\mathbf x)=\nabla u(\mathbf x)-\nabla u'(\mathbf x)$ have local H\"older continuity in ${B^{+}_{h_2 }}(\mathbf{y'})\cup{B^{+}_{h_2 }}(\mathbf{y''})$.
Thus the regularity conditions of Proposition~\ref{Lem:Near field to boundary} are satisfied for each choice of $w(\mathbf x)$ and we can assume $\alpha=1/2$.
Thus for $\forall \mathbf x\in \Gamma^{\pm}_h$ it implies
\begin{equation}\label{eq n to b}
		\sup_{\mathbf x\in\Gamma^{\pm}_h}|w(\mathbf x)|\leq C_f\left(\ln|\ln \varepsilon_1|\right)^{-1/2}.
\end{equation}
Combining \eqref{eq f to n} with \eqref{eq n to b}, it yields that 
\begin{equation*}
		\begin{aligned}
			\left| {\ln {\varepsilon _1}} \right|
			&\geq C\sqrt {\ln ({1 \over \varepsilon })} \geq C{\left(\ln \left({1 \over \varepsilon }\right)\right)^{{1 \over 4}}},
		\end{aligned}
	\end{equation*}
and
\begin{equation*}
	\begin{aligned}
		(\ln|\ln \varepsilon_1|)^{-1/2}
		&\leq {\left(\ln{\left(\ln \left({1 \over \varepsilon }\right)\right)^{{1 \over 4}}} \right)^{ - {1 \over 2}}} =2{\left(\ln \left(\ln \left({1 \over \varepsilon }\right)\right)\right)^{ - {1 \over 2}}},
	\end{aligned}
\end{equation*}
where $C>0$.
It should be remarked that when two different choices of $w(\mathbf x)$ are made, the corresponding constants $C_f,C$ are not the same, and here we take the largest of the two choices.
Therefore, when we fix $\mathbf x$ on $\Gamma^{\pm}_h$, there exists a positive constant $C_b$ such that \eqref{eq sup b} holds.

The proof is complete.
\end{proof}

%--------------------------------------------------------------------
\section{Micro-local analysis around a corner}\label{sec:micro}

This section is devoted to the quantitative behavior of the scattered field $u^s-u'^s$ around the corner $Q_h$ defined below using the introduced integral identity. 
Let us first specify our geometric notation. 
Under the assumption that $\mathbf y_2$ is a vertex of $K$ such that $\mathfrak{h}=d_M(K, K')=\dist(\mathbf{y_2}, K')$. 
Since $\Delta$ is invariant under rigid motion, without loss of generality, we assume that $\mathbf y_2=\mathbf 0$.
Assume that $h>0$ is sufficiently small such that
\begin{align}\label{eq:Qh}
	Q_h=B_h(\mathbf 0)\cap K,\ B_h\cap \overline {K'}=\emptyset.
\end{align}
Denote
\begin{align}\label{eq:gamma la}
	\partial Q_h\setminus \partial B_h=\Gamma_h^+\cup \Gamma_h^-,\ \Lambda_h=\partial Q_h \setminus (\Gamma_h^+\cup \Gamma_h^-).
\end{align}
Without loss of generality, by suitable rigid motion, we assume that $\Gamma_h^-$ lies on $x_1^+$-axis and $\Gamma_h^+$ posses an anti-clockwise direction to $\Gamma_h^-$ with $\angle(\Gamma_h^+,\Gamma_h^-)=\theta_0\in (0,\pi)$.
The exterior normal vectors to $\Gamma^+_{h}$ and $\Gamma^-_{h}$ are denote by  $\nu_1$ and $\nu_2$ respectively.
Therefore
\begin{align}
	\Gamma^+_{h}&=\{\mathbf x \in \mathbb R^2|\mathbf x=r(\cos {\theta _0},\sin {\theta _0}),\, r\in [0,h]\},\, \Gamma^-_{h}=\{\mathbf x \in \mathbb R^2~|~\mathbf x=r(1,0 ),\, r\in [0,h]\},\notag \\
	{\nu_1} &= ( - \sin {\theta _0},\cos {\theta _0}),\quad {\nu_2} = (0, - 1).\label{eq Gamma parameters }
\end{align}

%---------------------------------------------------------------------
\subsection{The complex geometric optical (CGO) solution}

We introduce the \textit{complex geometric optical solution (CGO)} for our subsequent analysis.
The detailed proof of the following Lemma is omitted.

\begin{lem}
Assuming $\tau,k\in \mathbb R_+$, we choose $\rho  =\rho(\tau,k)= \tau {\mathbf{d}} + {\rm{i}}\sqrt {{k^2} + {\tau ^2}} {{\mathbf{d}}^\perp }$, where  $\mathbf{d},\mathbf{d}^{\perp} \in \mathbb{S}^{1}$ satisfying $\mathbf{d}\perp \mathbf{d}^{\perp}$.
We define the CGO solution in the following form
\begin{equation}\label{CGO}
		u_{0}(\mathbf x)=e^{\rho\cdot \mathbf x},\quad \mathbf x\in \mathbb R^2.
\end{equation}
It yields that $\Delta u +k^2u=0$ in $\mathbb R^2.$ Let $Q_h$ be given by \eqref{eq:Qh}. 
There exist $\alpha'\in (0,1]$ and a open set $ \mathcal K_{\alpha'} \subset \mathbb S^1$ such that
\begin{align}\label{rem d}
		-\mathbf d \cdot {\bf \hat{x}}>\alpha'>0,\quad \forall \mathbf{d}\in \mathcal K_{\alpha'},\quad \forall {\mathbf x }  \in Q_h\backslash\{\mathbf 0\} ,\quad \bf \hat{x}=\mathbf x/|\mathbf x|.
\end{align}
\end{lem}

In Lemma~\ref{lem vo=N} we give local analytic behaviors of the total wave field $u'$ associated with $K'$ in $B_h$, where $B_h \cap K'=\emptyset$. Before that, we introduce the vanishing order of $u(\mathbf x)$ at $\mathbf x_0$, where $u(\mathbf x)$ is an analytic function in $B_h(\mathbf x_0)$ with $h\in \mathbb R_+$.  

\begin{defn}\cite[Definition 1.3]{cao2020nodal}\label{def:vanishing order}
	Let $u(\mathbf x)$ be analytic in $B_h(\mathbf x_0)$. 
	For a given point $\mathbf{x_0}\in B_h(\mathbf x_0)$, if there exists a number $N\in \mathbb{N}\cup \{0\}$ such that 
	\begin{equation}\label{Vanishing order}
		\mathop {\lim }\limits_{\rho \to  + 0} {1 \over {{\rho^m}}}\int\limits_{B_\rho(\mathbf{x_0})} {\left| {u(\mathbf x)} \right|\mathrm d\mathbf x}  = 0\ \mbox{for} \ m=0,1,\ldots ,N+1,
	\end{equation}
	we say that $u(\mathbf x)$ vanishes at $\mathbf{x_0}$ up to the order $N$. 
The largest possible number $N$ such that (\ref{Vanishing order}) is fulfilled is called the \textit{vanishing order} of $u(\mathbf x)$ at $\mathbf{x_0}$ and we write
$$
    {\rm Vani}(u; \mathbf x_0) = N.
$$
If (\ref{Vanishing order}) holds for any $N \in \mathbb{N}$, then we say that the vanishing order is infinity.
\end{defn}

\begin{lem}\label{lem vo=N}
Let $K$, $K'\in \mathcal{B}$ be admissible obstacles.
Suppose that $u(\mathbf x)$ and $u'(\mathbf x)$ are the total wave field to \eqref{impedance problem} associated with $K$ and $K'$ respectively.
The $u'(\mathbf x)$ satisfies the uniform bounded estimate $\left\Vert u'\right \Vert_{H^1(B_{R+1}\setminus K')}\leq \mathcal{E}$ defined by \eqref{eq:uniform bound}.
If \eqref{eq:Qh} is fulfilled, then $u'(\mathbf x)$ is analytic in $B_h$ and has a finite vanishing order $N\in \mathbb{N}\cup\{0\}$ at the origin.
Moreover, $u'(\mathbf x)$ has the following decomposition:
\begin{equation}\label{u' N}
	 u'(\mathbf x) = u'_{N}(\mathbf x)+\delta u'_{N+1}(\mathbf x),\quad \mathbf x=r(\cos \theta, \sin \theta)\in B_h, 
\end{equation}
where $u'_{N}(\mathbf x)=\left({a_N}{e^{\mathrm i N\theta }} + {b_N}{e^{ - \mathrm i N\theta }}\right){{{k^N}} \over {{2^N}N!}}{r^N}$ and
\begin{align*}
	\delta u'_{N+1}(\mathbf x)
	&= \left({a_N}{e^{\mathrm i N\theta }} + {b_N}{e^{ - \mathrm i N\theta }}\right){{{k^N}} \over {{2^N}N!}}\sum\limits_{p = 1}^\infty  {\left[ {{{{{( - 1)}^p}N!} \over {p!(N + p)!}}{{\left( {{k \over 2}} \right)}^{2p}}{r^{2p + N}}} \right]} \notag \\
	&\quad +\sum\limits_{n = N+1}^\infty  {\left({a_n}{e^{\mathrm i n\theta }} + {b_n}{e^{ - \mathrm i n\theta }}\right){J_n}(kr)}.\notag 
\end{align*}
Here $J_n(t)$ is the $n$-th Bessel function of the first kind given by
\begin{equation}\label{eq:Bessel }
		{J_n}(t) = \sum\limits_{p = 0}^\infty  {{{{{( - 1)}^p}} \over {p!(n + p)!}}{{\left( {{t \over 2}} \right)}^{n + 2p}}} 
\end{equation}
and $a_n,\ b_n \in \mathbb C$ with $n=N,N+1,\ldots$. 
Furthermore, if $h$ is sufficiently small such that
\begin{align}\label{eq:kh}
		kh<1,
\end{align}
it yields that
\begin{align}\label{eq:lem42}
		|u'_N(\mathbf x)|\leq C_{N}\ r^N,\quad| \delta u'_{N+1}(\mathbf x)|\leq \mathcal{R}\ r^{N+1}, \quad \forall \mathbf x=r(\cos\theta,\sin \theta)\in B_h,
\end{align}
	where the constants $C_{N}$ and $\mathcal{R}$ depend on a-priori parameters only.
\end{lem}

\begin{proof}
Since $u'(\mathbf x)\in H^1(\mathbb R^2 \backslash K' )$ is a solution to \eqref{impedance problem} associated with $K'$, by interior regularity of elliptic partial differential operator, one has $u'(\mathbf x)$ is analytic in $B_h$. 
By contradiction, assume that the vanishing order of $u'(\mathbf x)$ at $\mathbf 0$ is infinite, according to the Definition~\ref{def:vanishing order} of vanishing order, and Rellich's Lemma in \cite[Lemma 2.12]{colton2019inverse}, we can conclude that $u'(\mathbf x) \equiv 0$ in  $\mathbb R^2 \setminus  K'$. 
Thus the far-field measurement $u'^s(\mathbf x) \equiv 0$. 
   
In what follows we consider two separate cases for the vanishing order of $u'(\mathbf x)$ at $\mathbf 0$.
\medskip

 \noindent {\bf Case 1:}  ${\rm Vani}(u',\mathbf 0)=N\in \mathbb N$.
Using the spherical wave expansion of $u'$ (cf. \cite{colton2019inverse}), it arrives that
\begin{equation}\label{eq u' N}
		u'(\mathbf x) = ({a_N}{e^{\mathrm i N\theta }} + {b_N}{e^{ - \mathrm i N\theta }})J_N(kr)+\sum\limits_{n = N+1}^\infty  {({a_n}{e^{\mathrm i n\theta }} + {b_n}{e^{ - \mathrm i n\theta }}){J_n}(kr)},
\end{equation}
which is an absolutely uniform convergent series in $B_h$. 
]Here $a_i$ and $b_i$ are constants, $(a_N,b_N)\neq \mathbf 0$.    In view of \eqref{eq:Bessel }, we can rewrite \eqref{eq u' N} as
	\begin{align}\label{eq new un}
			u'(\mathbf{x})
			&= \left({a_N}{e^{\mathrm i N\theta }} + {b_N}{e^{ - \mathrm i N\theta }}\right){{{k^N}} \over {{2^N}N!}}{r^N} \\
			&+ \left({a_N}{e^{\mathrm i N\theta }} + {b_N}{e^{ - \mathrm i N\theta }}\right){{{k^N}} \over {{2^N}N!}}\sum\limits_{p = 1}^\infty  {\left[ {{{{{( - 1)}^p}N!} \over {p!(N + p)!}}{{\left( {{k \over 2}} \right)}^{2p}}{r^{2p + N}}} \right]} \notag \\
			&+\sum\limits_{n = N+1}^\infty  {\left({a_n}{e^{\mathrm i n\theta }} + {b_n}{e^{ - \mathrm i n\theta }}\right){J_n}(kr)}.\notag 
	\end{align}

In the following, we estimate the upper bound of coefficients $|a_n|$ and $|b_n|,\ n\geq  N$.
It is clear that
   \begin{equation*}
   	\left\Vert u^{\prime }\left( \mathbf{x} \right)  \right\Vert^{2}_{L^{2}\left( \partial B_{h}\right)  }  =\int_{\partial B_{h}} u^{\prime }\left( \mathbf{x} \right)  \overline{u^{\prime }\left( \mathbf{x} \right)  } \mathrm{d} \sigma =\sum^{+\infty }_{n=N} \left( a^{2}_{n}+b^{2}_{n}\right)  J^{2}_{n}\left(kh\right).
   \end{equation*}
Notice that the first zero $\{j_{n,1}\}_{n\in\mathbb N}$ and the first zero $\{j'_{n,1}\}_{n\in\mathbb N}$ of the Bessel function $\{J_n(t)\}_{n\in\mathbb N}$ and $\{J_n'(t)\}_{n\in\mathbb N}$ satisfy $n\leq j'_{n,1}\leq j_{n,1}$ for $\forall {n\in\mathbb N}$ and $J_0(0)=1,J_n(0)=0,n\in \mathbb N \setminus\{0\}$.
Then the Bessel function $\{J_n(t)\}_{n\in \mathbb N \setminus\{0\}}$ is monotonically increasing on the interval $(0,kh]$, if $h$ is sufficiently small.
Therefore, due to \eqref{eq:kh}, we can find a $t^*\in (0,kh]\subset (0,j_{1,1}]$ such that $|J_n(kh)|\geq |J_n(t^*)|=C_1$ for $\forall n\in\mathbb N$.   
Let $\sigma: H^1(B_h)\rightarrow L^2(\partial B_h)$ be a bounded linear trace operator, then the continuity of $\sigma$ and elliptic equation interior estimation implies that 
   $$
   { \|\sigma u'\|_{L^{2}(\partial B_h)}\leq C_2\|u'\|_{H^{1}(B_h )}} \leq C_2C_3\mathcal E,
   $$
   where $\mathcal E$ is the uniform bound of $u'$ defined by \eqref{eq:uniform bound} in Proposition~\ref{uniform L2 bound} and $C_3$ is a constant depending a-priori parameters only.
Therefore, there exists a constant $C_4$ such that
   \begin{equation*}
   	 \max\{|a_{n}|,|b_n|\}
   	 \leq \frac{\left\Vert u^{\prime }\left( \mathbf{x} \right)  \right\Vert_{L^{2}\left( \partial B_{h}\right)  }  }{|J_{n}\left( kh\right)  |} 
   	 \leq \frac{C_{2}C_3\mathcal{E} }{C_{1}} 
   	 \leq C_4,\quad
   	 \forall n\in\mathbb N.
   \end{equation*}
	
With the upper bound of $|a_n|,|b_n|$, we focus on the first term coefficient of \eqref{u' N} and it is not difficult to find that there exists a constant $C_N:=\left(|{a_N}|+ |{b_N}|\right){{{k^N}} \over {{2^N}N!}}$ such that
	\begin{equation*}
			|u'_N(\mathbf x)|\leq C_N\ r^N,
	\end{equation*}
where $u'_N(\mathbf x)=\left({a_N}{e^{\mathrm i N\theta }} + {b_N}{e^{ - \mathrm i N\theta }}\right){{{k^N}} \over {{2^N}N!}} $.
    
Furthermore, we consider the latter two terms of \eqref{u' N}.
Since the series of $J_n(t), $ converges absolutely and uniformly in any bounded domain (cf.\cite{korenev2002bessel}).
Thus we can find two constants $C_5,C_6,C_7$ that satisfy
	\begin{align*}
    &\nonumber \Big|\left({a_N}{e^{\mathrm i N\theta }} + {b_N}{e^{ - \mathrm i N\theta }}\right){{{k^N}} \over {{2^N}N!}}\sum\limits_{p = 1}^\infty  {\left[ {{{{{( - 1)}^p}N!} \over {p!(N + p)!}}{{\left( {{k \over 2}} \right)}^{2p}}{r^{2p + N}}} \right]}\\ \nonumber
    &\qquad +\sum\limits_{n = N+1}^\infty  {\left({a_n}{e^{\mathrm i n\theta }} + {b_n}{e^{ - \mathrm i n\theta }}\right){J_n}(kr)}\Big| \\ \nonumber
    &\leq r^{N+1}\ \left(|{a_N}| + |{b_N}|\right){{{k^N}} \over {{2^N}N!}}\left|\sum\limits_{p = 1}^\infty  {\left[ {{{N!} \over {p!(N + p)!}}{{\left( {{k \over 2}} \right)}^{2p}}{r^{2p - 1}}} \right]}\right|\\
    &\qquad+r^{N+1}\ \sum\limits_{n = N + 1}^\infty  {\left(|{a_n}| + |{b_n}|\right)} \left|\sum\limits_{p = 0}^\infty  {{{{{( - 1)}^p}} \over {p!(n + p)!}}{{\left( {{kr \over 2}} \right)}^{n + 2p-N-1}}}\right|\\ \nonumber
    &\leq C_5(|a_N|+|b_N|)\ r^{N+1}+C_6\sum_{n=N+1}^{\infty}(|a_n|+|b_n|)\ r^{N+1}\nonumber\\
    &\leq (2C_4C_5+C_6C_7)\ r^{N+1}      \nonumber\\
    &\leq \mathcal R\ r^{N+1},\nonumber
    \end{align*}	
where $\mathcal R:=2C_4C_5+C_6C_7$ is a positive constant depending on a-priori parameters.
For simplicity, one has 
   \begin{align*}
     |\delta u'_{N+1}(\mathbf x)|\leq \mathcal R\ r^{N+1},
    \end{align*}	
where 
   \begin{align*}
    \delta u'_{N+1}(\mathbf x)= &\left({a_N}{e^{\mathrm i N\theta }} + {b_N}{e^{ - \mathrm i N\theta }}\right){{{k^N}} \over {{2^N}N!}}\sum\limits_{p = 1}^\infty  {\left[ {{{{{( - 1)}^p}N!} \over {p!(N + p)!}}{{\left( {{k \over 2}} \right)}^{2p}}{r^{2p + N}}} \right]}\\ \nonumber
    &\qquad +\sum\limits_{n = N+1}^\infty  {\left({a_n}{e^{\mathrm i n\theta }} + {b_n}{e^{ - \mathrm i n\theta }}\right){J_n}(kr)}. 
   \end{align*}
   
   \medskip
 \noindent {\bf Case 2:}  ${\rm Vani}(u',\mathbf 0)=0$. Since $u'$ is analytic in $B_h$, we have $u'(\mathbf x)=u'(\mathbf 0)+\frac{\partial u'(\zeta \mathbf x)}{\partial x_1} x_1+ \frac{\partial u'(\zeta \mathbf x)}{\partial x_2} x_2$, where $u'(\mathbf 0)\neq 0$ and $\zeta \in (0,1)$. 
Following the similar argument for Case 1, we can obtain \eqref{eq:lem42}. 
   	
   	\medskip
The proof is complete.
\end{proof}

%---------------------------------------------------------------------
\subsection{Integral identity}

The main purpose of this subsection is to perform a micro-local analysis near a vertex $\mathbf y_2=\mathbf 0$ of the obstacle $K$ by establishing an integral identity in Proposition~\ref{pro:51}. 
Since the impedance parameter $\eta(\mathbf x) $ associated with the admissible obstacle $K\in \mathcal  B$ fulfills $\eta(\mathbf x)\in \Xi_{\mathcal B}$, one has
\begin{align}\label{eq: estimation of eta }
	\eta(\mathbf x)=\eta(\mathbf 0) +\delta \eta(\mathbf x),\quad |\delta \eta| \leq \| \eta\|_{C^1(\Gamma_h^\pm  )} |\mathbf x|\leq M_2 |\mathbf x|.	
\end{align}

\begin{prop}\label{pro:51}
Let $K,K'\in\mathcal{B}$ and $u(\mathbf x),u'(\mathbf x)$ be the solution to \eqref{impedance problem} associated with obstacles $K$ and $K'$, respectively.
Let $\eta(\mathbf x),\eta'(\mathbf x)\in\Xi_{\mathcal B}\subset C^1(B_R)$ be the impedance parameters of $K,K'$ respectively.
Let $u_0(\mathbf x)$ be the CGO solution defined by \eqref{CGO}. 
Assume that $Q_h$ is defined by \eqref{eq:Qh}.
Recall that $u'(\mathbf x)$ is analytic in $Q_h$ and has the expansion \eqref{eq u' N} in $Q_h$. 
Let 
\begin{equation}\label{eq Gamma+_}
\begin{aligned}
		&\Gamma^+=\{\mathbf x \in \mathbb R^2|\mathbf x=r(\cos {\theta _0},\sin {\theta _0}),\, r\in [0,\infty) \},\\
		&\Gamma^-=\{\mathbf x \in \mathbb R^2~|~\mathbf x=r(1,0 ),\, r\in [0,\infty) \}.
\end{aligned}
\end{equation}
Then the following integral identity holds that
	\begin{align}\label{integral equation}
		&\int_{{\Gamma ^ \pm }} {{\partial {u_0}(\mathbf x)} \over {\partial \nu }} u'_{N}(\mathbf x)\mathrm d\sigma
		=\int_{{\Gamma ^ \pm_h }} {{u_0}(\mathbf x){{\partial [u'(\mathbf x) - u(\mathbf x)]} \over {\partial \nu }}} \mathrm d\sigma
		-\eta(\mathbf 0) \int_{{\Gamma ^ \pm_h }} {{u_0}(\mathbf x)[u (\mathbf x)- u'(\mathbf x)]} \mathrm d\sigma\nonumber\\
		&- \int_{{\Gamma ^ \pm_h }} \delta\eta(\mathbf x){{u_0}(\mathbf x)[u (\mathbf x)- u'(\mathbf x)]} \mathrm d\sigma
		-\eta(\mathbf 0) \int_{{\Gamma ^ \pm_h }} {{u_0}(\mathbf x)} u'_{N}(\mathbf x)\mathrm d\sigma
		- \int_{{\Gamma ^ \pm_h }} \delta\eta(\mathbf x){{u_0}(\mathbf x)} u'_{N}(\mathbf x)\mathrm d\sigma\nonumber  \\ 
		&-\eta(\mathbf 0)\int_{{\Gamma ^ \pm_h }} \delta{u'_{N+1}} (\mathbf x){u_0}(\mathbf x)\mathrm d\sigma
		-\int_{{\Gamma ^ \pm_h }} \delta\eta(\mathbf x)\delta{u'_{N+1}} (\mathbf x){u_0}(\mathbf x)\mathrm d\sigma\nonumber
		 -\int_{{\Gamma ^ \pm_h }} \delta{u'_{N+1}(\mathbf x){{\partial {u_0}(\mathbf x)} \over {\partial \nu }}} \mathrm d\sigma\nonumber \\
		&+\int_{{\Gamma^\pm\setminus\Gamma^\pm_h }} {{\partial {u_0}(\mathbf x)} \over {\partial \nu }} u'_{N}(\mathbf x)\mathrm d\sigma+\int_{\Lambda_h}  \left[{{u_0(\mathbf x)}{{\partial u'(\mathbf x)} \over {\partial \nu }} - u'(\mathbf x)} {{\partial {u_0}(\mathbf x)} \over {\partial \nu }}\right]\mathrm d\sigma.
	\end{align}
where $\Gamma_h^\pm$ and $\Lambda_h$ are defined in \eqref{eq:gamma la}.
\end{prop}

\begin{proof}
By Green's second formula, we get
	\begin{equation}\label{Green formula}
		\int_{{Q_h}} {(\phi\Delta \psi - \psi} \Delta \phi)\mathrm{d}\mathbf x
		=\int_{\Gamma ^ \pm_h} {(\phi{{\partial \psi} \over {\partial \nu }}}  - \psi{{\partial \phi} \over {\partial \nu }})\mathrm {d}\sigma+\int_{\Lambda_h} {(\phi{{\partial \psi} \over {\partial \nu }}}  - \psi{{\partial \phi} \over {\partial \nu }})\mathrm {d}\sigma,
	\end{equation}
where $\phi=u_0$, $\psi=u'$.
Since both $u'$ and $u_0$ satisfy the Helmholtz equation $\Delta u+k^2u=0$ in $Q_h$ and \begin{equation}\label{decompose v}
		\begin{cases}
			\displaystyle \psi=u'=u'-u+u\\
			\displaystyle{{\partial \psi} \over {\partial \nu }} = {{\partial u'} \over {\partial \nu }} = {{\partial (u' - u)} \over {\partial \nu }} +   {{\partial u} \over {\partial \nu }},
		\end{cases}
	\end{equation}
substituting \eqref{decompose v} into \eqref{Green formula}, and using the impedance boundary conditions of $u$ on $\Gamma_h^\pm$ and \eqref{eq u' N}.
After some calculations, we obtain the integral identity \eqref{integral equation}.
\end{proof}

\begin{lem}\label{Lem:Gamma}
 For any real number $b > 0$ and any complex number $\mu$ satisfying $\mathfrak{R}\mu>1$, where $\Gamma(x)$ stands for the Gamma function and $\mathfrak{R}\mu$ stands the real part of $\mu$.
 If $\mathfrak{R}\mu>\mu_0>1$, where the $\mu_0$ is a constant depending $b>0,h>0$, then we have the estimate
 	\begin{equation}\label{eq gamma}
 		\left| {\int_0^h {{r^{b - 1}}{e^{ - \mu r}}} \mathrm {d}r} \right| \le \left|{{\Gamma (b)} \over {{\mu ^b}}}\right| + {{2{e^{{{ - {\mathop{\mathfrak {R}}\nolimits} \mu h} \over 2}}}} \over {{\mathop{\mathfrak {R}}\nolimits} \mu }}.
 	\end{equation}
 \end{lem}
 \begin{proof}
 	
 By Laplace transform, it yields that
 $$
 	\int_0^h {{r^{b - 1}}{e^{ - \mu r}}} \mathrm {d}r = \int_0^\infty  {{r^{b - 1}}{e^{ - \mu r}}} \mathrm {d}r - \int_h^\infty  {{r^{b - 1}}{e^{ - \mu r}}} \mathrm {d}r.
 $$
 	
 If $\mathfrak{R}\mu \geq 2(b-1)/h=\mu_0>1$, then ${r^{b - 1}} \le {e^{{{\mathfrak R\mu r} \over 2}}}$, for all $r\geq h$, hence we have
 \begin{equation}\label{eq h-infty gamma}
 		\left| {\int_h^\infty  {{r^{b - 1}}{e^{ - \mu r}}\mathrm d r} } \right| \le \int_h^\infty  {{r^{b - 1}}{e^{ - \mathfrak{R}\mu r}}\mathrm d r}  \le \int_h^\infty  {{e^{ - {{\mathfrak{R}\mu r} \over 2}}}\mathrm d r}  = {{2{e^{ - {{\mathfrak{R}\mu h} \over 2}}}} \over {\mathfrak{R}\mu }}.
 \end{equation}

The proof is complete.
\end{proof}

\begin{prop}\label{prop right2}
Let $K, K'\in \mathcal{B}$ be admissible obstacles, $u(\mathbf x),u'(\mathbf x)$ be the solution to impedance scattering problems \eqref{impedance problem} associated with $K$ and $K'$, respectively.
Let $\eta(\mathbf x),\eta'(\mathbf x)\in\Xi_{\mathcal B}\subset C^1(B_R)$ be the impedance parameters of $K,K'$ respectively.
For $Q_h$ defined by \eqref{eq:Qh}, suppose that $\mathbf d$ in the CGO solution $u_0$ given by \eqref{CGO} fulfills \eqref{rem d}.
Let	${\rm Vani}(u'; \mathbf 0) = N$, where $u'(\mathbf x)$ has the expansion \eqref{eq u' N} in $Q_h$.
If \eqref{eq:far-field error} in Theorem~\ref{mian result} is satisfied and the parameter $\tau$ in the CGO solution $u_0$ given by \eqref{CGO} fulfilling
	\begin{align}\label{eq:tau cond}
		\tau>2(N+1)/h,
	\end{align}
then it yields that
			\begin{align}
			\left\vert	\int_{{\Gamma ^ \pm_h }} {{\partial {u_0}(\mathbf x)} \over {\partial \nu }} u'_{N}(\mathbf x)\mathrm d\sigma\right\vert
			&\le {C_1}\mathrm{T}(\varepsilon)h 
			+ {C_2}\mathrm{T}(\varepsilon)h 
			+ {C_3}\mathrm T(\varepsilon)h 
			+ {C_4}\left({1 \over {{\tau ^{N + 1}}}} + {{{e^{ - {{\tau h} \over 2}}}} \over \tau }\right)\notag \\
			&+ {C_5}\left({1 \over {{\tau ^{N + 2}}}} + {{{e^{ - {{\tau h} \over 2}}}} \over \tau }\right)
			+ {C_6}\left({1 \over {{\tau ^{N+2}}}} + {{{e^{ - {{\tau h} \over 2}}}} \over \tau }\right)
			\notag \\
			&+ {C_7}\left({1 \over {{\tau ^{N+3}}}} + {{{e^{ - {{\tau h} \over 2}}}} \over \tau }\right)+ {C_8}\left({1 \over {{\tau ^{N+1}}}} + {{{e^{ - {{\tau h} \over 2}}}}}\right)
			+ {C_9}{e^{ - {h \over 2}\tau }}\notag \\
			&
			+ {C_{10}} \tau{h^{{1 \over 2}}}{e^{ - \alpha '\tau h}}, \label{eq:int inequality}
		\end{align}
where the constants $\alpha'$ is given  in \eqref{rem d} and  positive constants $C_j(j=1,2,...,10),$ depend the a-priori parameters only.
Here $\mathrm T(\varepsilon)=C_b(\ln\ln (1/\varepsilon))^{-\frac{1}{2}}$, where $C_b$ is a positive constant depending on the a-priori parameters only given in \eqref{eq sup b}.
\end{prop}

\begin{proof}
By \eqref{integral equation}, to prove \eqref{eq:int inequality}, we only need to estimate each term on the right-hand side of (\ref{integral equation}).
According to \eqref{rem d}, one has
    \begin{equation}\label{eq:u0 bound}
	\left| {{u_0}(\mathbf x)} \right| \le {e^{{\mathop{\rm \Re }\nolimits} \rho  \cdot \mathbf x}} \le {e^{ - \alpha '\tau r}}\leq 1,
    \end{equation}
where $0<\alpha'<1$ is given in \eqref{rem d}.
For the estimation of impedance parameters $\eta(\mathbf x)$, the constants $M_1$ and $M_2$ come from \eqref{eq:xi} and \eqref{eq: estimation of eta } respectively.
Combining \eqref{eq:u0 bound} with Proposition~\ref{Lem:Boundary estimation}, we can get
    \begin{align}
	\left\vert \int_{{\Gamma ^ \pm_h }} {{u_0}(\mathbf x){{\partial [u'(\mathbf x) - u(\mathbf x)]} \over {\partial \nu }}} \mathrm d\sigma \right\vert \label{right 1}
	&\leq \sqrt {\sigma ({\Gamma_h ^ \pm })} {\left({\int_{{\Gamma_h ^ \pm }} {\left| {{u_0}(\mathbf x){{\partial (u' - u)} \over {\partial \nu }}} \right|} ^2}\mathrm d\sigma\right)^{{1 \over 2}}} \notag \\
	&\leq \sqrt {2h} \left| {{u_0}(\mathbf x)} \right|\sqrt {2h}{\left\| {\nabla (u - u')} \right\|_{{L^\infty }({\Gamma_h ^ \pm })}} \notag \\
	&\leq 2h{\left\| {\nabla (u - u')} \right\|_{{L^\infty }({\Gamma_h ^ \pm })}} \leq C_1\mathrm T(\varepsilon)h,
	\end{align}

	\begin{align}
    \left\vert- \eta(\mathbf 0) \int_{{\Gamma_h ^ \pm }} {{u_0}(\mathbf x)[u(\mathbf x) - u'(\mathbf x)]} \mathrm d\sigma \right\vert
    &\leq \left| \eta(\mathbf 0)  \right|\sqrt {\sigma ({\Gamma_h ^ \pm })} {\left(\int_{{\Gamma_h ^ \pm }} {{{\left| {{u_0}(\mathbf x)[u(\mathbf x) - u'(\mathbf x)]} \right|}^2}} \mathrm d\sigma\right)^{{1 \over 2}}}  \notag \\
    &\leq \sqrt {2h} \left| {{u_0}(\mathbf x)} \right|M_1 \sqrt {2h}{\left\| {u - u'} \right\|_{{L^\infty }({\Gamma_h ^ \pm })}} \\
    &\leq 2 M_1h{\left\| {u - u'} \right\|_{{L^\infty }({\Gamma_h ^ \pm })}} \leq {C_2}\mathrm T(\varepsilon)h.\notag
    \end{align}
and
    \begin{align}
    \left\vert- \int_{{\Gamma ^ \pm_h }} \delta\eta(\mathbf x){{u_0}(\mathbf x)[u (\mathbf x)- u'(\mathbf x)]} \mathrm d\sigma \right\vert
    &\leq \left| \delta\eta(\mathbf x)  \right|\sqrt {\sigma ({\Gamma_h ^ \pm })} {\left(\int_{{\Gamma_h ^ \pm }} {{{\left| {{u_0}(\mathbf x)[u(\mathbf x) - u'(\mathbf x)]} \right|}^2}} \mathrm d\sigma\right)^{{1 \over 2}}} \notag  \\
    &\leq \sqrt {2h} \left| {{u_0}(\mathbf x)} \right|M_2 \sqrt {2h}{\left\| {u - u'} \right\|_{{L^\infty }({\Gamma_h ^ \pm })}}\\
    &\leq 2 M_2 h{\left\| {u - u'} \right\|_{{L^\infty }({\Gamma_h ^ \pm })}} \leq {C_3}\mathrm T(\varepsilon)h.\notag
    \end{align}
Using \eqref{eq:tau cond}, by virtue of \eqref{eq gamma} in   Lemma~\ref{Lem:Gamma}, we have the following estimates
    \begin{align}
    \left| 	-\eta(\mathbf 0) \int_{{\Gamma ^ \pm_h }} {{u_0}(\mathbf x)} u'_{N}(\mathbf x)\mathrm d\sigma\right|
    &\leq2 \left| \eta(\mathbf 0)  \right|C_{N}\int_0^h {{r^{N}}{e^{ - \alpha '\tau r}}\mathrm d r}\\
    &\leq 2M_1C_{N}\left[{{\Gamma (N + 1)} \over {{{(\alpha '\tau )}^{N + 1}}}} + {{2{e^{ - {{\tau h} \over 2}}}} \over {\alpha '\tau }}\right]\notag \\
    &\leq {C_4}\left({1 \over {{\tau ^{N + 1}}}} + {{{e^{ - {{\tau h} \over 2}}}} \over \tau }\right), \notag 
    \end{align}
    
    \begin{align}
    \left| 	- \int_{{\Gamma ^ \pm_h }} \delta\eta(\mathbf x){{u_0}(\mathbf x)} u'_{N}(\mathbf x)\mathrm d\sigma\right|
    &\leq2 M_2 C_{N}\int_0^h {{r^{N+1}}{e^{ - \alpha '\tau r}}\mathrm d r}\\
    &\leq 2M_2 C_{N}\left[{{\Gamma (N + 2)} \over {{{(\alpha '\tau )}^{N + 2}}}} + {{2{e^{ - {{\tau h} \over 2}}}} \over {\alpha '\tau }}\right]\notag \\
    &\leq {C_5}\left({1 \over {{\tau ^{N + 2}}}} + {{{e^{ - {{\tau h} \over 2}}}} \over \tau }\right), \notag 
    \end{align}
    
    \begin{align}
    \left| -\eta(\mathbf 0)\int_{{\Gamma ^ \pm_h }} \delta{u'_{N+1}} (\mathbf x){u_0}(\mathbf x)\mathrm d\sigma\right|
    &\leq2 \left| \eta(\mathbf 0)  \right|\mathcal{R}\int_0^h {{r^{N+1}}{e^{ - \alpha '\tau r}}\mathrm d r}  \\
    &\leq 2M_1\mathcal{R}\left[{{\Gamma (N + 2)} \over {{{(\alpha '\tau )}^{N + 2}}}} + {{2{e^{ - {{\tau h} \over 2}}}} \over {\alpha '\tau }}\right]\notag \\
    &\leq {C_6}\left({1 \over {{\tau ^{N+2}}}} + {{{e^{ - {{\tau h} \over 2}}}} \over \tau }\right), \notag 
   \end{align}
   
   \begin{align}
    \left| -\int_{{\Gamma ^ \pm_h }} \delta\eta(\mathbf x)\delta{u'_{N+1}} (\mathbf x){u_0}(\mathbf x)\mathrm d\sigma\right|
    &\leq2 M_2\mathcal{R}\int_0^h {{r^{N+2}}{e^{ - \alpha '\tau r}}\mathrm d r}  \\
    &\leq 2M_2\mathcal{R}\left[{{\Gamma (N + 3)} \over {{{(\alpha '\tau )}^{N + 3}}}} + {{2{e^{ - {{\tau h} \over 2}}}} \over {\alpha '\tau }}\right]\notag \\
    &\leq {C_7}\left({1 \over {{\tau ^{N+3}}}} + {{{e^{ - {{\tau h} \over 2}}}} \over \tau }\right), \notag 
   \end{align}
and 
   \begin{align}\label{eq:no number1}
    \left| -\int_{{\Gamma ^ \pm_h }} \delta{u'_{N+1}(\mathbf x){ \frac{\partial {u_0}(\mathbf x)}{\partial \nu }}} \mathrm d\sigma\right|
    &\leq 2\mathcal{R}\sqrt {{k^2} + 2{\tau ^2}} \int_0^h {{r^{N+1}}{e^{ - \alpha '\tau r}}\mathrm d r} \\
    &\leq 2\mathcal{R}\sqrt {{k^2} + 2{\tau ^2}}\left[{{\Gamma (N + 2)} \over {{{(\alpha '\tau )}^{N + 2}}}} + {{2{e^{ - {{\tau h} \over 2}}}} \over {\alpha '\tau }}\right]\notag \\
    &\leq C_8\left({1 \over {{\tau ^{N+1}}}} + {{{e^{ - {{\tau h} \over 2}}}}}\right). \notag
    \end{align}
    By virtue of \eqref{eq h-infty gamma}, we can obtain
    \begin{align}\label{eq:no number2}
	\left|\int_{{\Gamma ^ \pm }\backslash \Gamma _h^ \pm } {{{\partial {u_0}({\bf{x}})} \over {\partial \nu }}} {u'_N}({\bf{x}}){\rm{d}}\sigma\right|
	&\leq 2C_N\sqrt{k^2+2\tau^2}  {2 \over \tau }{e^{ - {h \over 2}\tau }} \\
	&\leq C_9{e^{ - {h \over 2}\tau }}. \notag
    \end{align}

Finally, we estimate the boundary integral over arc $\Lambda_h$. Direct calculations, it yields that
    \begin{align}\notag
    {\left\| {{u_0}} \right\|_{{H^1}(\Lambda_h )}}  \le \sqrt {1 + 2{\tau ^2} + {k^2}} \theta _0^{{1 \over 2}}{h^{{1 \over 2}}}{e^{ - \alpha '\tau h}},	\ {\left\| {{{\partial {u_0}} \over {\partial \nu }}} \right\|_{{L^2}(\Lambda_h )}}  \le \sqrt {2{\tau ^2} + {k^2}} \theta _0^{{1 \over 2}}{h^{{1 \over 2}}}{e^{ - \alpha '\tau h}},
    \end{align}
where $\theta_0$ is the opening angle of $Q_h$. Hence, by virtue of the trace theorem, one has
    
    \begin{align}\label{right 6}
    \left\vert\int_{\Lambda_h}  {{u_0}{{\partial u'} \over {\partial \nu }} - u'} {{\partial {u_0}} \over {\partial \nu }}\mathrm d\sigma\right\vert
    &\leq {\left\| {{u_0}} \right\|_{{H^{{1 \over 2}}}(\Lambda_h )}}{\left\| {{{\partial u'} \over {\partial \nu }}} \right\|_{{H^{ - {1 \over 2}}}(\Lambda_h )}} + {\left\| {u'} \right\|_{{L^2}(\Lambda_h )}}{\left\| {{{\partial {u_0}} \over {\partial \nu }}} \right\|_{{L^2}(\Lambda_h )}}\notag \\
    &\leq {\left\| {{u_0}} \right\|_{{H^1}(\Lambda_h )}}{\left\| {u'} \right\|_{{H^1}({Q_h})}} + {\left\| {u'} \right\|_{{H^1}({Q_h})}}{\left\| {{{\partial {u_0}} \over {\partial \nu }}} \right\|_{{L^2}(\Lambda_h )}}\notag \\
    &\leq {\left\| {u'} \right\|_{{H^1}({Q_h})}}\left({\left\| {{u_0}} \right\|_{{H^1}(\Lambda_h )}} + {\left\| {{{\partial {u_0}} \over {\partial \nu }}} \right\|_{{L^2}\left(\Lambda_h \right)}}\right)\notag \\
    &\leq 2\mathcal M\sqrt {1 + 2{\tau ^2} + {k^2}} \theta _0^{{1 \over 2}}{h^{{1 \over 2}}}{e^{ - \alpha '\tau h}} \leq C_{10} \tau{h^{{1 \over 2}}}{e^{ - \alpha '\tau h}}.
    \end{align}
    
Using \eqref{right 1} to \eqref{right 6} we obtain \eqref{eq:int inequality}.
\end{proof}

%-----------------------------------------------------------------

In Proposition~\ref{prop left 2}, we give a lower bound for the integral defined in \eqref{eq:pro 53}, which plays an important role in the proof of Theorem~\ref{mian result}.  

\begin{prop}\label{prop left 2}
Suppose that $u'$ is the solution to \eqref{impedance problem} associated with the obstacle $K'$ and $\Gamma^\pm$ are defined in \eqref{eq Gamma+_}.
Let $u_0(\mathbf x)$ be the CGO solution defined by \eqref{CGO} and $u'_N(\mathbf x)$ be given by \eqref{u' N}, where $N$ is the vanishing order of $u'$ at $\mathbf 0$ with $N\in \mathbb N\cup \{0\}$.
	
If 
\begin{align}\label{eq:tau pro53}
		\tau>\max({2(N+1)}/{h},k,\tau_0), 
\end{align}
where $\tau_0\in \mathbb R_+$ is sufficiently large,  then it yields that
\begin{equation}\label{eq:pro 53}
	\left| {\int_{\Gamma^\pm } {{{\partial {u_0}({\bf{x}})} \over {\partial \nu }}} {{u'}_N}({\bf{x}}){\rm{d}}\sigma } \right| \ge 
		{C_{(N,k,{\theta _0})}}\frac{1} {\tau ^N},\quad N\in \mathbb N\cup \{0\}  , 
\end{equation}
where the positive constant $C_{(N,k,\theta_0)}$ depends on the a-priori parameters only.
\end{prop}

\begin{proof}
Recall that $\Gamma^\pm$ are given in \eqref{eq Gamma+_}. Hence, we can choose $\mathbf{d} = \left(\cos\varphi,\sin\varphi \right)$ and $\mathbf{d^ \bot } = \left(-\sin\varphi,\cos\varphi\right)$ with $\varphi\in (\theta_0+\frac{\pi}{2},\frac{3\pi}{2})$ such that $\mathbf d$ fulfills \eqref{rem d}. 
The unit exterior normal vectors $\nu_1$ and $\nu_2$ to $\Gamma^+$ and $\Gamma^-$ are given by \eqref{eq Gamma parameters } respectively. 
Therefore, it can be obtained that
   \begin{align}\label{eq:normal vector}
    	\frac{\partial u_{0}}{\partial \nu }\Big|_{\nu=\nu_1} 
	&=\tau \left[\sin \left( \varphi -\theta_{0} \right)  +\mathrm{i} \sqrt{1+\frac{k^{2}}{\tau^{2} } } \cos \left( \varphi -\theta_{0} \right)  \right]u_{0},\\
	\frac{\partial u_{0}}{\partial \nu}\Big|_{\nu=\nu_2}
	&=\tau\left(-\sin\varphi-\mathrm i\sqrt{\frac{k^2}{\tau^2}+1}\cos\varphi\right) u_{0}.   \nonumber
    \end{align}
   
Denote
    \begin{align}\label{eq C'N C''N}
    	&C^{\prime }_{N}=\left( {a_{N}}  {e^{\mathrm{i} N\theta_{0} }}  +{b_{N}}  {e^{-\mathrm{i} N\theta_{0} }}  \right)  {\frac{{{k^{N}}  }  }{{{2^{N}}  N!}  } }  ={\frac{{{k^{N}}  }  }{{{2^{N}}  N!}  } }  \left( \begin{matrix}a_{N}&b_{N}\end{matrix} \right)  \left( \begin{matrix}e^{\mathrm{i} N\theta_{0} }\\ e^{-\mathrm{i} N\theta_{0} }\end{matrix} \right):=\mathbf a^\top \mathbf b_+  ,\nonumber \\ 
    	&C^{\prime \prime }_{N}=\left( {a_{N}}  +{b_{N}}  \right)  {\frac{{{k^{N}}  }  }{{{2^{N}}  N!}  } }  ={\frac{{{k^{N}}  }  }{{{2^{N}}  N!}  } }  \left( \begin{matrix}a_{N}&b_{N}\end{matrix} \right)  \left( \begin{matrix}1\\ 1\end{matrix} \right):=\mathbf a^\top \mathbf b_-,
    \end{align}      
where $\mathbf a=(a_N, b_N)^\top $, $\mathbf b_+=(e^{\mathrm{i} N\theta_{0} }, e^{-\mathrm{i} N\theta_{0} })^\top $, and  $\mathbf b_-=(1, 1)^\top $. 
Recall that $u'_N(\mathbf x)$ is given in \eqref{eq u' N}.
It yields that
        \begin{equation}\label{eq C'C''}
    	u'_N(\mathbf{x})=C'_Nr^N \mbox{ with } \mathbf{x}\in\Gamma^+,\quad 
    	u'_N(\mathbf{x})=C''_Nr^N \mbox{ with }   \mathbf{x}\in\Gamma^-, 
    	    \end{equation}   
where $|a_N|+|b_N|\neq 0$.
    
By Laplace transform, we have
	\begin{equation}\label{eq Laplace trans}
		\begin{aligned}
			\int_{\Gamma^+}r^Ne^{\rho \cdot \mathbf{\hat{x}_+}r}\mathrm d \sigma
			&={{\Gamma (N+1)} \over { [- (\tau \mathbf d + \mathrm i\sqrt {{k^2} + {\tau ^2}} {\mathbf d^ \bot }) \cdot (\cos{\theta _0},\sin{\theta _0})]^{N+1}}},\\
			\int_{\Gamma^-}r^Ne^{\rho \cdot \mathbf{\hat{x}_-}r}\mathrm d \sigma
			&={{\Gamma (N+1)} \over [{ - (\tau  \mathbf d + \mathrm i\sqrt {{k^2} + {\tau ^2}} {\mathbf d^ \bot }) \cdot (1,0)}]^{N+1}},
		\end{aligned}
	\end{equation}
where  $\hat{\mathbf{x} }_{+}=[\cos\theta_0,\sin \theta_0]$ and $ \hat{\mathbf{x} }_{-}=[1,0]$.
For simplicity, we denote 
	\begin{align*}
		\mathbf z_1
		&:=-\tau \left[\cos \left( \varphi -\theta_{0} \right)  +\mathrm{i} \sqrt{\frac{k^{2}}{\tau^{2}}+1 } \sin \left( \varphi -\theta_{0} \right)\right], \\
		\mathbf z_2
		&:=\tau \left[\sin \left( \varphi -\theta_{0} \right)  +\mathrm{i} \sqrt{\frac{k^{2}}{\tau^{2} }+1 } \cos \left( \varphi -\theta_{0} \right)  \right],\\
		\mathbf z_3
		&:=\tau\left(-\cos\varphi+\mathrm i \sqrt{\frac{k^2}{\tau^2}+1} \sin \varphi\right), \
		\mathbf z_4
		:=\tau\left(-\sin\varphi-\mathrm i\sqrt{\frac{k^2}{\tau^2}+1}\cos\varphi\right).       	\end{align*}
		
Let us denote 
\begin{align}\label{eq:C1C2}
	C_1:=\mathbf z_2\mathbf z_3^{N+1}, \ C_2:=\mathbf z_4\mathbf z_1^{N+1}.
\end{align}
Combining \eqref{eq:normal vector}, \eqref{eq C'C''} with \eqref{eq Laplace trans}, one has
    \begin{align}\label{eq Ctau}
    &\int_{{\Gamma ^ + }} {{{\partial {u_0}({\bf{x}})} \over {\partial \nu }}} {u'_N}({\bf{x}}){\rm{d}}\sigma  + \int_{{\Gamma ^ - }} {{{\partial {u_0}({\bf{x}})} \over {\partial \nu }}} {u'_N}({\bf{x}}){\rm{d}}\sigma\\
    &= C'_N\int_{\Gamma^+} {\mathbf z_2 e^{\rho \cdot \mathbf{\hat{x}_+}r}{r^N}{\rm{d}}\sigma}
     +C'{'_N}\int_{\Gamma^-}\tau{e^{\rho \cdot \mathbf{\hat{x}_-}r}}{r^N}{\rm{d}}\sigma\nonumber\\
    &={\frac{{\Gamma (N+1)}  C^{\prime }_{N}{\mathbf z_2}  }{{[-(\tau \mathbf{d} +\mathrm{i} \sqrt{{k^{2}}  +{\tau^{2} }  } {\mathbf{d}^{\bot } }  )\cdot (\cos {\theta_{0} }  ,\sin {\theta_{0} }  )]^{N+1}}  } }
    +{\frac{{\Gamma (N+1)}  C^{\prime }{^{\prime }_{N}  }  \mathbf z_4 }{[{-(\tau \mathbf{d} +\mathrm{i} \sqrt{{k^{2}}  +{\tau^{2} }  } {\mathbf{d}^{\bot } }  )\cdot (1,0)}  ]^{N+1}} } \nonumber \\
    &={\Gamma (N+1)}\frac{{ {{{C^\prime_{N} \mathbf z_2  \mathbf z_3}^{N+1}  }  +{{C^{\prime \prime }_{N}\mathbf z_4 \mathbf z_1}^{N+1}  }  } }}{{{{{\left( {\mathbf z_1\mathbf z_3}  \right)  }^{N+1}  }  }  }} 
    ={\Gamma (N+1)}  \frac{{{k^{N}}  }  }{{{2^{N}}  N!}  } \frac{{{{{\mathbf{a}^{\top } C_{1}\mathbf{b}_{+} }  }  +\mathbf{a}^{\top } C_{2}\mathbf{b}_{-}   }  }  }  {{{{{\left( {\mathbf{z}_{1} \mathbf z_3}  \right)  }^{N+1}  }  }  }  } \nonumber\\
    &={\Gamma (N+1)}\frac{{{k^{N}}  }  }{{{2^{N}}  N!}  } \frac{\mathbf z}{\mathbf w},\nonumber
    \end{align}
 where
    $$
     \mathbf z={{{{{\mathbf{a}^{\top } C_{1}\mathbf{b}_{+} }  }  +\mathbf{a}^{\top } C_{2}\mathbf{b}_{-} {}  }  }  }       ,\quad     \mathbf w= {{{{\left( \mathbf z_1 \mathbf z_3  \right) }^{N+1}  }  }  }.
    $$
    
We first derive the lower bound of $\mathbf z$ in \eqref{eq Ctau}. 
We shall distinguish three distinct cases.

\medskip 
\noindent {\bf Case 1:}  ${\rm Vani}(u',\mathbf 0)=N\in \mathbb N$, and $\theta_0\neq\frac{\pi}{2N}+\frac{\ell \pi}{N}, \ell\in\{1,\ldots, N-1\}.$

It can be verified that 
    \begin{align*}
    |\mathbf z|^2
    &=\mathbf{a}^{H} \left( |C_{1}|^{2}\overline{\mathbf{b} }_{+} \mathbf{b}^{\top }_{+} +|C_{2}|^{2}\overline{\mathbf{b} }_{-} \mathbf{b}^{\top }_{-} \right)  \mathbf{a} \\
    &\quad + \mathbf{a}^{H} \left( \overline{C}_{1} \overline{\mathbf{b} }_{+} C_{2}\mathbf{b}^{\top }_{-} +\overline{C}_{2} \overline{\mathbf{b} }_{-} C_{1}\mathbf{b}^{\top }_{+} \right)  \mathbf{a} \\
    &\geq  \mathbf{a}^{H} \left( |C_{1}|^{2}\overline{\mathbf{b} }_{+} \mathbf{b}^{\top }_{+} +|C_{2}|^{2}\overline{\mathbf{b} }_{-} \mathbf{b}^{\top }_{-} \right) \mathbf{a}.
    \end{align*}
        
Let $\lambda_1$ and $\lambda_2$ be two real eigenvalues of $ |C_{1}|^{2}\overline{\mathbf{b} }_{+} \mathbf{b}^{\top }_{+} +|C_{2}|^{2}\overline{\mathbf{b} }_{-} \mathbf{b}^{\top }_{-} $. 
Through direct calculation we have
    $$
    \lambda_{1} =|C_{1}|^{2}+|C_{2}|^{2}+\sqrt{|C_{1}|^{4}+|C_{2}|^{4}+|C_{1}|^{2}|C_{2}|^{2}\left( 1-\cos^{2} 2N\theta_{0} +2\cos 2N\theta_{0} \right)  } ,
    $$
and
    $$
    \lambda_{2} =|C_{1}|^{2}+|C_{2}|^{2}-\sqrt{|C_{1}|^{4}+|C_{2}|^{4}+|C_{1}|^{2}|C_{2}|^{2}\left( 1-\cos^{2} 2N\theta_{0} +2\cos 2N\theta_{0} \right)  }.
    $$
Under the assumption, it implies that $\lambda_1>\lambda_2>0$. 
Therefore we have
    \begin{align*}
    	|\mathbf z|^2\geq \lambda_2 (|a_N|^2+|b_N|^2).
    \end{align*}
We shall deduce the asymptotic lower bound for $\lambda$ when $\tau\rightarrow \infty$. 
By the definition of $C_1$ and $C_2$, there exists a choice of $\varphi\in(\theta_0+\frac{\pi}{2},\frac{3\pi}{2})$ such that $|C_1|^2\neq|C_2|^2$, without loss of generality, we assume $|C_1|>|C_2|$. 
Therefore we can deduce that
    \begin{align}\label{eq:upbound C}
    	&\sqrt{|C_{1}|^{4}+|C_{2}|^{4}+|C_{1}|^{2}|C_{2}|^{2}\left( 1-\cos^{2} 2N\theta_{0} +2\cos 2N\theta_{0} \right)  }  \\
    	&\leq |C_1|^2\sqrt{3-\cos^22N\theta_0+2\cos2N\theta_0}\notag \\
    	&= \tau^{2N+4}( [1+\frac{k^{2}}{\tau^{2} } \cos^{2} (\varphi -\theta_{0} )][1+\frac{k^{2}}{\tau^{2} } \sin^{2} \varphi ]^{N+1})\sqrt{3-\cos^22N\theta_0+2\cos2N\theta_0}. \notag
    \end{align}
It yields that 
    \begin{align*}
    	|C_1|^2+|C_2|^2 &=\left[ \tau^{2} +k^{2}\cos^{2} (\varphi -\theta_{0} \right]  (\tau^{2} +k^{2}\sin^{2} \varphi )^{N+1} \\
    	&\quad +\left( \tau^{2} +k^{2}\cos^{2} \varphi \right)  [\tau^{2} +k^{2}\sin^{2} (\varphi -\theta_{0} )]^{N+1}\notag  \\
    	&=\tau^{2N+4} \{ [1+\frac{k^{2}}{\tau^{2} } \cos^{2} \left( \varphi -\theta_{0} \right)  ](1+\frac{k^{2}}{\tau^{2} } \sin^{2} \varphi )^{N+1}\\
    	&\quad +( 1+\frac{k^{2}}{\tau^{2} } \cos^{2} \varphi)  [1+\frac{k^{2}}{\tau^{2} } \sin^{2} (\varphi -\theta_{0} )]^{N+1}\}.
    \end{align*}
Combining the above inequality and \eqref{eq:upbound C}, we obtain that $\lambda_2\geq C_{(N, k, \theta_0)}(\tau)\tau^{2N+4}$, where
   \begin{align*}
    C_{(N, k, \theta_0)}(\tau)
    &:= [1+\frac{k^{2}}{\tau^{2} } \cos^{2} \left( \varphi -\theta_{0} \right)  ](1+\frac{k^{2}}{\tau^{2} } \sin^{2} \varphi )^{N+1}
    	+( 1+\frac{k^{2}}{\tau^{2} } \cos^{2} \varphi)  [1+\frac{k^{2}}{\tau^{2} } \sin^{2} (\varphi -\theta_{0} )]^{N+1}\\
    &\quad -( [1+\frac{k^{2}}{\tau^{2} } \cos^{2} (\varphi -\theta_{0} )][1+\frac{k^{2}}{\tau^{2} } \sin^{2} \varphi ]^{N+1})\sqrt{3-\cos^22N\theta_0+2\cos2N\theta_0}.	  
    \end{align*}
Since  $\theta_0\neq\frac{\pi}{2N}+\frac{\ell \pi}{N}, \ell\in\{1,\ldots, N-1\}$, it follows that 
    \begin{equation}\label{eq:case1}
    \sqrt{3-\cos^22N\theta_0+2\cos2N\theta_0}<2.
    \end{equation}
Taking the limit $\tau\rightarrow\infty$, we have $C_{(N, k, \theta_0)}(\tau)\rightarrow C>0$, where the positive constant $C>0$ depends a-priori parameters $N, \theta_0$ only. 
Hence there exists a constant $\tau_1$. 
If $\tau>\tau_1$, then one has that 
    \begin{align}\label{eq:535 lower}
	|\mathbf z|^2\geq C_{(N,\theta_0)}^2\tau^{2N+4}, 
    \end{align}
where the constant $C_{(N,\theta_0)}^2:=C(|a_N|+|b_N|)>0$ depends the a-priori parameters $N,\ \theta_0$ only.
    
    \medskip
    \noindent {\bf Case 2:}  ${\rm Vani}(u',\mathbf 0)=N\in \mathbb N$, and $\theta_0=\frac{\pi}{2N}+\frac{\ell \pi}{N}, \ell\in \{1,\ldots, N-1\}$.

%Under the assumptions of this case, the \eqref{eq:case1} no longer holds, but we would like to prove that the constant $C_{(N,k,\theta_0)}$ is non-zero and satisfies \eqref{eq:535 lower}.

In view of $C_1,C_2$ given by \eqref{eq:C1C2}, since $\theta_0,N,k$ are fixed, we can  write $C_1(\varphi)$ and $C_2(\varphi)$, where $\varphi \in (\theta_0+\frac{\pi}{2},\frac{3\pi}{2})$ is the polar angle of $\mathbf d$. According to the definitions of $\mathbf a,\mathbf b_+,\mathbf b_-$ given in \eqref{eq C'N C''N}, using \eqref{eq Ctau} and $\theta_0=\frac{\pi}{2N}+\frac{\ell \pi}{N}, \ell\in \{1,\ldots, N-1\}$, it directly yields that 
$$
\mathbf{z} =\left( \mathrm{i} C_{1}+C_{2},\  C_{1}+C_{2}\right)  \left( \begin{gathered}a_{N}\\ b_{N}\end{gathered} \right)=g_1C_1(\varphi)+g_2C_2(\varphi) =g(\varphi), % =\left( f_{1},\  f_{2}\right)  \left( \begin{gathered}a_{N}\\ b_{N}\end{gathered} \right):=\mathbf f^\top\mathbf a.
$$ 
where $g_1=b_N+\mathrm{i} a_N $ and $g_2=a_N+\mathrm{i}  b_N$ are two fixed complex numbers.

  Taking $\tau\rightarrow\infty$,  one has
  $$
  \lim_{\tau \rightarrow \infty } g\left( \varphi \right)  =g_{2}-g_{1}\mathrm{i} e^{\mathrm{i} \frac{\pi }{N} }e^{-2\mathrm{i} \left( N+1\right)  \varphi }:=f(\varphi).
  $$  
  By calculation, it is easy to see that when $|g_1|\neq|g_2|$, one has $f(\varphi)\neq0$.  When $|g_1|=|g_2|$, it is not difficult to see that $f(\varphi)$ has only finitely number of zeros. Therefore, there exists an angle $\varphi_0 \in (\theta_0+\frac{\pi}{2},\frac{3\pi}{2})$ such that $f(\varphi_0) \neq 0$. 
 
 Hence when there exist a sufficient large constant $\tau_2$, if $\tau>\tau_2$, one can conclude that there exists an angle  $\varphi_0 \in (\theta_0+\frac{\pi}{2},\frac{3\pi}{2})$ such that $|\mathbf z|\geq |f(\varphi_0) |/2> 0$. The remaining argument is similar for  Case 1 and one can obtain \eqref{eq:535 lower}. 
 \medskip
 
 Let $\tau_0=\max\{\tau_1,\tau_2\}$. For Case 1 and Case 2, when $\tau>\tau_0$, we prove \eqref{eq:pro 53}.

%  Therefore, in this case, we can directly repeat the argument of Case 1 and obtain an identical result.
    
    \medskip
    \noindent{\bf Case 3:} ${\rm Vani}(u',\mathbf 0)=0.$ In fact, we have $u'(\mathbf 0)\neq 0.$   
    
    According to \eqref{eq C'N C''N}, one has $C^{\prime }_N=C^{\prime \prime }_{N}=\left( {a_{N}}  +{b_{N}}  \right)  {\frac{{{k^{N}}  }  }{{{2^{N}}  N!}  } }\neq 0$. By direct calculations, 
    \begin{align*}
    \mathbf z
    &=C^\prime_{N} \mathbf z_2 \mathbf z_3^{N+1} +C^{\prime \prime }_{N} \mathbf z_4 {\mathbf z_1}^{N+1}  \\
    &=C'_N\tau^2\{-\sin \left( 2\varphi -\theta_{0} \right)  -\left( \frac{k^{2}}{\tau^{2} } +1\right)  \sin \theta_{0} -2\mathrm{i} \sqrt{\frac{k^{2}}{\tau^{2} } +1} \cos \left( \varphi -\theta_{0} \right)  \cos \varphi \},
    \end{align*}
    and 
    $$
    |\mathbf z|^2\geq \tau^44(\frac{k^2}{\tau^2}+1)\cos^2 \left( \varphi -\theta_{0} \right)  \cos^2 \varphi>\tau^4[4\cos^2 \left( \varphi -\theta_{0} \right)  \cos^2 \varphi].
    $$
    Due to $\varphi\in(\theta_0+\frac{\pi}{2},\frac{3\pi}{2})$, similarly, one has a constant $C_{(N,\theta_0)}$ such that $|\mathbf z|^2>C^2_{(N,\theta_0)}\tau^4$.

    \medskip
    
   For $\mathbf w$ given by  \eqref{eq Ctau}, if $\tau>k$ we can find a constant $C_k$ such that
 \begin{align}\label{eq:537 upper}
 	&\left| {{{{\left\{ {-\tau \left[\cos \left( \varphi -\theta_{0} \right)  +\mathrm{i} \sqrt{\frac{k^{2}}{\tau^{2}}+1 } \sin \left( \varphi -\theta_{0} \right)\right]\tau\left(-\cos\varphi+\mathrm i \sqrt{\frac{k^2}{\tau^2}+1} \sin \varphi\right)}  \right\}  }^{N+1}  }  }  }  \right|  \nonumber\\
	&=\tau^{2N+2}\left| { \left[2+2\frac{k^2}{\tau^2}\sin^2(\varphi-\theta_0)\right] }^{N+1} \right|  \\
	&\leq C_k\tau^{2N+2},\nonumber
 \end{align}
where $N\in \mathbb N\cup \{0\}.$

Under two admissible conditions, we have a uniformed lower bound \eqref{eq:pro 53} from \eqref{eq:535 lower} and \eqref{eq:537 upper}, where 
$$
C_{(N, k ,\theta_0)}=
 \Gamma(N+1)\frac{{{k^{N}}  }  }{{{2^{N}}  N!}  } \frac{C_{(N,\theta_0)}}{C_k} ,\quad N\in \mathbb N\cup \{ 0 \} .
$$

   The proof is complete.
\end{proof}

%------------------------------------------------------------------
\section{Proof of Theorem~\ref{mian result}  }\label{sec:proof}
   
\begin{proof}[Proof of Theorem \ref{mian result}]
Since $\Delta$ is invariant under rigid motion, according to the geometrical setup specified at the beginning of Section~\ref{sec:micro}, we assume that $\mathbf 0$ is the vertex of $K$ such that $\mathfrak h = \dist(\mathbf 0, K')$, where $\mathfrak h$ is the Hausdorff distance between $K$ and $K'$.  
Correspondingly, let us denote the total wave associated with the obstacle $K'$ by $u'(\mathbf x)$. 
Denote the $L^{\infty}$-norm of the difference of the far-field patterns by $\varepsilon = {\left\| u^{\infty }_{\left( K,\eta \right)}  (\hat{\mathbf{x} } ) - {u^{\infty }_{\left( K',\eta^{\prime } \right)}(\hat{\mathbf{x} } )} \right\|_{{L^\infty }({\mathbb{S}^1})}} $. 
    
Recall that $Q_h$, $\Gamma^\pm_{h}$, and $\Lambda_h$ are defined in \eqref{eq:Qh} and \eqref{eq:gamma la} respectively.  
It is noted that the parameter $h$ needs to satisfy the conditions of \eqref{eq:cond h} in Lemma~\ref{Lem:local Holder continuous} and \eqref{eq:kh} in Lemma~\ref{lem vo=N}.
Set  
$$
h=\min\left\{\underline \ell,\ \mathfrak h,\ \frac{1}{k+1}\right\},
$$  
where $\underline\ell$ is defined by Definition~\ref{def:Class B} and $k$ is the wavenumber of the fixed incident wave \eqref{eq incident wave}. Hence \eqref{eq:kh} is fulfilled. 
At the end of the proof, we shall show that there exists a positive number $\varepsilon_0$ depending on a-prior parameters only so that  \eqref{eq:cond h} is also satisfied when the far-field error satisfies  $\varepsilon <\varepsilon_0 $. 
       
By noting \eqref{eq:cond h}, due to Proposition \ref{Lem:Boundary estimation}, we can get the boundary estimate 
\begin{equation}\label{eq error T}  
    \sup_{\partial \Gamma^\pm_{h}}(|u-u'|+|\nabla(u-u')|)\leq    \mathrm T(\varepsilon):= C_b(\ln\ln (1/\varepsilon))^{-\frac{1}{2}}, 
\end{equation}
when $\varepsilon < \varepsilon_m$. 
Here $C_b$ and $\varepsilon_m$ depend only on the a-priori parameters.
    
In the following, we distinguish two separate cases. 

\medskip
\noindent{\bf Case 1:}  ${\rm Vani}(u',\mathbf 0)=N\in \mathbb N$.
    
By Proposition~\ref{pro:51}, we have the integral identity \eqref{integral equation}. 
Since $u'$ is analytic in $B_h$, we first assume that the vanishing order $u'$ at $\mathbf 0$ is $N\in \mathbb N$, which implies that $u'$ has the decomposition \eqref{u' N}. When the vanishing order of $u'$ at $\mathbf 0$ is $N=0$, we can prove the stability result \eqref{eq:stability function} by Case 2 below.
   
Let the CGO solution $u_0$ be defined in \eqref{CGO} with a positive parameter $\tau$, where $\mathbf d$ fulfills \eqref{rem d}.     
Furthermore, by the definition of admissible class $\mathcal B$ and  Proposition \ref{uniform L2 bound}, one has $\norm{u(\mathbf x)}_{H^1(B_{2R} \backslash K )}\leq \mathcal{E}$ and $\norm{u'(\mathbf x)}_{H^1(B_{2R} \backslash K' )} \leq \mathcal{E}$, where $u$ is the total wave field associated with $K$.  
If $\tau$ fulfills the condition \eqref{eq:tau pro53} in Proposition~\ref{prop left 2}, using Proposition~\ref{prop right2}, it yields that 
\begin{align}
	C\left\vert	\int_{{\Gamma ^ \pm_h }} {{\partial {u_0}(\mathbf x)} \over {\partial \nu }} u'_{N}(\mathbf x)\mathrm d\sigma\right\vert
	&\le \mathrm{T}(\varepsilon)h 
			+ \mathrm{T}(\varepsilon)h 
			+ \mathrm T(\varepsilon)h 
			+ \left({1 \over {{\tau ^{N + 1}}}} + {{{e^{ - {{\tau h} \over 2}}}} \over \tau }\right) \label{eq:62} \\
	&+\left({1 \over {{\tau ^{N + 2}}}} + {{{e^{ - {{\tau h} \over 2}}}} \over \tau }\right)
			+ \left({1 \over {{\tau ^{N+2}}}} + {{{e^{ - {{\tau h} \over 2}}}} \over \tau }\right)
			+ \left({1 \over {{\tau ^{N+3}}}} + {{{e^{ - {{\tau h} \over 2}}}} \over \tau }\right)\notag \\
	&+\left({1 \over {{\tau ^{N+1}}}} + {{{e^{ - {{\tau h} \over 2}}}}}\right)
			+ {e^{ - {h \over 2}\tau }}
			+\tau{h^{{1 \over 2}}}{e^{ - \alpha '\tau h}}, \notag
\end{align}
where $C$ is a positive constant depending on a-prior parameters only. If $\tau$ satisfies \eqref{eq:tau pro53}, using Proposition~\ref{prop left 2}, it yields that 
\begin{equation}\label{eq:63}
    \left\vert	\int_{{\Gamma ^ \pm_h }} {{\partial {u_0}(\mathbf x)} \over {\partial \nu }} u'_{N}(\mathbf x)\mathrm d\sigma\right\vert
    \geq C_{(N,k,\theta_0)}{1 \over {{\tau ^N}}},
\end{equation}
where  $C_{(N,k,\theta_0)}$ is a positive constant $C_{(N,k,\theta_0)}$ depending  on the a-priori parameters only. 
Combining \eqref{eq:62} with \eqref{eq:63}, we know that there exist a positive constant $C$ such that
    \begin{equation}\label{eq:64}
	\begin{aligned}
		C\frac{1}{\tau^N}
			&\le \mathrm{T}(\varepsilon)h 
			+ \mathrm{T}(\varepsilon)h 
			+ \mathrm T(\varepsilon)h 
			+ \left({1 \over {{\tau ^{N + 1}}}} + {{{e^{ - {{\tau h} \over 2}}}} \over \tau }\right)\\
			&+\left({1 \over {{\tau ^{N + 2}}}} + {{{e^{ - {{\tau h} \over 2}}}} \over \tau }\right)
			+ \left({1 \over {{\tau ^{N+2}}}} + {{{e^{ - {{\tau h} \over 2}}}} \over \tau }\right)
			+ \left({1 \over {{\tau ^{N+3}}}} + {{{e^{ - {{\tau h} \over 2}}}} \over \tau }\right)\\
			&+\left({1 \over {{\tau ^{N+1}}}} + {{{e^{ - {{\tau h} \over 2}}}}}\right)
			+ {e^{ - {h \over 2}\tau }}
			+\tau{h^{{1 \over 2}}}{e^{ - \alpha '\tau h}}.
	\end{aligned}
    \end{equation}
It is ready to know that 
	\begin{align}\label{eq e ineq 1}
		&\exp(-t) \leq {1\over t},\quad \exp( - t) \leq {{(n + 4)!} \over {{t^{n + 4}}}},\quad \forall t>0 \ \mbox{and} \ n\in \mathbb N.
	\end{align}
   
Using $h\leq 1$, $\tau\geq 1$ and \eqref{eq e ineq 1}, after some calculations, from \eqref{eq:64} we have
\begin{equation*}
	\begin{aligned}
		C
		&\le {{\tau ^N}}\mathrm T(\varepsilon)h 
		+ {{\tau ^N}}\mathrm T(\varepsilon)h 
		+ {{\tau ^N}}\mathrm T(\varepsilon)h 
		+ \left({1 \over {{\tau}}} + {\tau ^{ - 1}}{h^{ - N}}\right)\\
		&+ \left({1 \over {{\tau ^{2}}}} + {\tau ^{ - 1}}{h^{ - N}}\right) 
		+\left({1 \over {{\tau ^{2}}}} + {\tau ^{ - 1}}{h^{ - N}}\right)
		+\left({1 \over {{\tau ^{3}}}} + {\tau ^{ - 1}}{h^{ - N}}\right) \\ 
		&+ \left({1 \over {{\tau}}} +  {\tau ^{ - 1}}{h^{ - (N + 1)}}\right)
		 + {\tau ^{ - 1}}{h^{{1 \over 2} - (N + 2)}}+ {\tau ^{ - 1}}{h^{ - (N + 1)}},
	\end{aligned}
\end{equation*}
    which can be used to derive that
\begin{equation}\label{eq new final}
	C \le {\tau ^N}\mathrm T(\varepsilon )h + {\tau ^{ - 1}}{h^{ - (N + 2)}}.
\end{equation}
    We emphasize that \eqref{eq new final} is valid when $\tau$ satisfies \eqref{eq:tau pro53}. 
    We estimate the right-hand side of \eqref{eq new final} by setting $\tau = \tau_e$ with
\begin{equation}\label{eq:taue}
	{\tau _e} = \left(\mathrm T (\varepsilon )\right)^{ - \frac{1} {N + 1}  }{h^{ - {{N + 3} \over {N + 1}} }}. 
\end{equation}
Therefore we obtain that 
    \begin{equation}\label{eq:bound 69}
	C \le \left( \mathrm T(\varepsilon )\right)^{ {1 \over {N + 1}} }{h^{ - {{{N^2} + 2N - 1} \over {N + 1}} }},
    \end{equation}
where $C$ is a positive constant depending on a-prior parameters only. 
    
In the following, we show that the choice \eqref{eq:taue} of $\tau$ can fulfill the condition \eqref{eq:tau pro53} when $\varepsilon$ is sufficiently small. 
By the definition of $\mathrm T(\varepsilon) $, for given positive numbers $N$, $h$, $k$ and $\tau_0$, we can show that $\left(\mathrm T(\varepsilon) \right)^{- {1\over N+1} }>\max\{{2(N+1)}/{h},k,\tau_0\}$ for $\varepsilon \in (0,  \varepsilon_{0,1} )$, where $\varepsilon_1\in \mathbb R^+$ depends on a-prior parameters only. 
Hence one has 
$$
  \tau_e \geq \tau_e {h^{{N+3\over N+1} }} =\left(\mathrm T(\varepsilon) \right)^{- {1\over N+1} }>\max\{{2(N+1)}/{h},k,\tau_0\}.
$$
    
From \eqref{eq:bound 69}, it can be deduced that
$$
  h \leq C \mathrm T(\varepsilon)^{1\over N^2+2N-1}.
$$
 Hence there exists a positive number $\varepsilon_{0,2}$ depending on a-prior parameters only. 
 If $0<\varepsilon<\varepsilon_{0,2} $, we conclude that $ h$ fulfills the condition \eqref{eq:cond h}. 
 By the definition of $\mathrm T(\varepsilon)$, for a given positive parameters $\underline \ell$ and $k$, there exists $\varepsilon_{0,3} \in \mathbb R_+$ depending on a-prior parameters only such that $ \mathfrak h \leq \min\left\{\underline \ell,  \frac{1}{k+1}\right\} $. 
Therefore it yields that    
\begin{equation}\label{eq:stabilityN}
	\mathfrak h\leq C \left(\ln\ln
	\frac{1}{\varepsilon}\right)^{-1\over 2(N^2+2N-1)}.
\end{equation}
Let $\varepsilon_0=\min\{ \varepsilon_{0,1}, \varepsilon_{0,2},\varepsilon_{0,3}\}$. 
We prove that if $\varepsilon\in (0,  \varepsilon_0 )$ then \eqref{eq:stability function} holds, where $\kappa=1/2\times (N^2+2N-1)$. 
    
    \medskip

\noindent{\bf Case 2:}  ${\rm Vani}(u',\mathbf 0)=0$.
Since the proof is similar to Case 1 we only give the necessary modifications in the following.  
    
Under the Assumption ${\rm Vani}(u;\mathbf 0)=0$, which implies that
$$
  u^{\prime }(\mathbf{x} )=u^{\prime }(\mathbf{0} )+\delta {u^{\prime }}  (\mathbf{x} ),\  \left| {\delta u^{\prime }(\mathbf{x} )}  \right|  \leq {\left\| \delta u^{\prime }\right\|_{{C^{1}}  \left( \overline{\Gamma^{\pm }_{h} } \right)  }  }  {\left| {\mathbf{x} }  \right|  }  ,\quad \mathbf x \in \Gamma_h^\pm,
$$
where $u^{\prime }(\mathbf{0} )\neq 0.$
	
By virtue of \eqref{integral equation}, one has
	\begin{align*}
		&u'(\mathbf 0)\int_{{\Gamma ^ \pm }} {{{\partial {u_0}(\mathbf x)} \over {\partial \nu }}} \mathrm d\sigma
		=\int_{{\Gamma ^ \pm_h }} {{u_0(\mathbf x)}{{\partial [u'(\mathbf x) - u(\mathbf x)]} \over {\partial \nu }}} \mathrm d\sigma
		-\eta(\mathbf 0) \int_{{\Gamma ^ \pm_h }} {{u_0(\mathbf x)}[u(\mathbf x) - u'(\mathbf x)]} \mathrm d\sigma\\
		&\quad -\int_{{\Gamma ^ \pm_h }} \delta\eta(\mathbf x){{u_0(\mathbf x)}(u - u')} \mathrm d\sigma
		-\eta(\mathbf 0) u'(\mathbf 0)\int_{{\Gamma ^ \pm_h }} {{u_0}(\mathbf x)} \mathrm d\sigma
		- u'(\mathbf 0)\int_{{\Gamma ^ \pm_h }}\delta\eta(\mathbf x) {{u_0}(\mathbf x)} \mathrm d\sigma\\
		&\quad -\eta(\mathbf 0)\int_{{\Gamma ^ \pm_h }} \delta{u'} (\mathbf x){u_0}(\mathbf x)\mathrm d\sigma  
		-\int_{{\Gamma ^ \pm_h }} \delta\eta(\mathbf x)\delta{u'} (\mathbf x){u_0}(\mathbf x)\mathrm d\sigma  
	    -\int_{{\Gamma ^ \pm_h }} \delta{u'(\mathbf x){{\partial {u_0}(\mathbf x)} \over {\partial \nu }}} \mathrm d\sigma\\
		&\quad +u'(\mathbf 0)\int_{{\Gamma^\pm\setminus\Gamma^\pm_h}} {{{\partial {u_0}(\mathbf x)} \over {\partial \nu }}} \mathrm d\sigma
		 +\int_{\Lambda_h} \left( {{u_0(\mathbf x)}{{\partial u'(\mathbf x)} \over {\partial \nu }} - u'} {{\partial {u_0}(\mathbf x)} \over {\partial \nu }} \right)\mathrm d\sigma.\\
	\end{align*}
Utilizing a similar proof for Proposition~\ref{prop right2}, we have the extra estimate
\begin{equation}\label{eq:7 estimates}
	\left| {\int_{\Gamma _h^ \pm } {{u_0}(\bf x)} \mathrm d\sigma} \right|
	\leq \int_{\Gamma _h^ \pm } {\left| {{u_0}(\bf x)} \right|} \mathrm d\sigma
	\leq 2\int_0^h {{e^{ - \alpha '\tau r}}} \mathrm d r
	\leq C\left({1 \over \tau } + {1 \over \tau }{e^{ - \alpha '\tau h}}\right).
\end{equation}

Using Proposition~\ref{prop left 2} and \eqref{eq:7 estimates}, adopting a similar argument for \eqref{eq new final}
we obtain that
	\begin{equation*}
		C \le  \tau\mathrm T(\varepsilon)h + \tau^{-1} h^{-{7\over 2}}.
	\end{equation*}
	We set $\tau = \tau_e$ with
	\begin{equation*}\label{eq:tau e}
		{\tau _e} = {\mathrm T(\varepsilon)^{-{1\over 2}}}{h^{-{9\over 4}}}.
	\end{equation*}
	 Using the fact $\tau>1$ and $h<1$, when $\varepsilon \in  (0, \varepsilon_0 )$ with $\varepsilon_0\in \mathbb R_+$ depending on a-prior parameters only, we have the following estimate
	\begin{equation*}
		C \le \mathrm T(\varepsilon)^{1\over 2}h^{-{4\over 5}}.
	\end{equation*}
which can be used to derive that
	\begin{equation}\label{eq: stability 0}
		\mathfrak h \leq C \left(\ln\ln
		\frac{1}{\varepsilon}\right)^{-\frac{1}{5}}.
	\end{equation}
	
	\medskip
	
Using the stability results \eqref{eq:stabilityN} and \eqref{eq: stability 0}, we can characterize the parameter $\kappa$ in \eqref{eq:stability function} as 
\begin{equation*}
		\kappa=
		\begin{cases}
			\dfrac{1}{2(N^2+2N-1)},\quad &N\in\mathbb N , \\[10pt] 
			\dfrac{1}{5},\quad &N\equiv 0,
	    \end{cases}
\end{equation*}
where $N$ is the vanishing order of $u'(\mathbf x)$ at $\mathbf 0$. 
	
The proof is complete.
\end{proof}

\medskip

%-----------------------------------------------------------------
\subsection{The implication of vanishing order to the stability estimate}\label{subsec:51}

In the last subsection of this section, we would like to discuss the relationship between the stability results \eqref{eq:stability function}  and the vanishing order $N$ of the function $u'(\mathbf x)$.

We have proved that the Hausdorff distance of two admissible obstacles can be controlled by a function of their errors about the far-field measurements, which we denote as $\mathfrak{g}(\varepsilon,N)=(\ln |\ln (1/\varepsilon)|)^{-\kappa(N)}$. 
It is easy to see that the exponential part of the stability function $\mathfrak{g}(\varepsilon,N)$ is related to the vanishing order $N$ of the function $u'(\mathbf x)$, where  
$$
\kappa(N)=\dfrac{1}{2(N^2+2N-1)},\quad N\geq1.
$$
It is clear that as $N\rightarrow \infty$, $\mathfrak{g}(\varepsilon,N)$ tends to 1. 
Therefore, the larger $N$ implies a poorer stability result of the Hausdorff distance $\mathfrak{h}$. 
This is consistent with the physical institution: the stronger scattering one has from the corners of the obstacle, the more stable reconstruction one should be able to produce.

%------------------------------------------------------------------
\section{Proof of Theorem~\ref{th: eta}}\label{sec:proof2}

In order to prove Theorem~\ref{th: eta}, we need the following three lemmas.
%\begin{lem}\label{lem:61}
%	Let $u(\mathbf x)$ be the solution of the impedance direct scattering problem~\eqref{impedance problem} associated with the obstacle $K\in\mathcal B$.
%	Then there exists a constant $\alpha$, $0<\alpha<1$, such that for every $R$ defined in Definition~\ref{def:Class B} one has $u(\mathbf x)\in C^\alpha(B_{R+1}\setminus K)$, and there exists a constant $C_R>0$ depending on the a-priori parameters, on $R$ only, such that 
%	\begin{equation}\label{eq:61}
%		\|u(\mathbf x) \|_{C^\alpha(B_{R+1}\setminus K) }\leq C_R.
%	\end{equation}
%	Furthermore, one has the estimation 
%	\begin{equation}\label{eq:62}
%		\|u(\mathbf x) \|_{H^1(\partial K)}\leq C,
%	\end{equation}
%	where the constant $C$ depending the a-priori parameters only.
%	And more importantly, we can demonstrate that the two upper bound estimations above can be uniformly established in the class $\mathcal B$.
%	\end{lem}
%	\begin{proof}
%		The estimation \eqref{eq:61} is a more or less standard regularity estimate up to the boundary.
%		We just point out that the estimates \eqref{eq:61} and \eqref{eq:62} can be obtained under the external impedance problem and only require the Lipchitz regularity of the $\partial K$, for details (cf.\cite[Theorem 3.2, Theorem 3.4]{alessandrini2013stable}) respectively.
		
%		To demonstrate that the upper bound estimation can be based on the Mosco convergence of the solution space holds uniformly in the admissible class $\mathcal B$, we use a contrary argument similar to the proof of Proposition~\ref{uniform L2 bound}.
%		Therefore, we just omit the details.
%	\end{proof}
	
	\begin{lem}\label{lem:62}
		Let $u(\mathbf x)$ be the solution of the impedance direct scattering problem~\eqref{impedance problem} associated with the obstacle $K\in\mathcal B$.
		For any $K\in\mathcal B$ and any point $\mathbf x_0\in\partial K$, and any $\delta>0$ satisfying $B_\delta(\mathbf x_0)\setminus K \subset B_{R+1}\setminus K$, there exists a positive constant $\mathcal E_L>0$ such that 
		\begin{equation}
			\left\Vert u\right\Vert_{L^{2}\left( B_\delta(\mathbf x_0)\setminus K \right)  }  \geq \mathcal E_L
		\end{equation}
		holds uniformly, where $\mathcal E_L$ is a positive constant depending on a-prior parameters only. 
	\end{lem}
	\begin{proof}
Using the contrary argument, we assume that the uniform lower bound does not exist. Hence  there is a subsequence $\{K_n\}_{n\in\mathbb{N}}\in\mathcal B$ converges to $K\in \mathcal B$ with respect to Hausdorff distance satisfying
$$
\Vert u_n \Vert_{L^2(B_{\delta_n}(\mathbf x_{0_n})\setminus K_n)}=a_n, \quad \forall n\in \mathbb N,\ \text{and}\ \lim_{n\rightarrow \infty}a_n=0,
$$
where $\mathbf x_{0_n}\in\partial K_n \rightarrow \mathbf x_0\in\partial K $ and $\delta_n\rightarrow \delta$ with $\delta_n \in \mathbb R_+$ and $\delta \in \mathbb R_+$. By Rellich's lemma in \cite[Lemma 2.12]{colton2019inverse} and unique continuation principle,we immediately get the contradiction.  
	\end{proof}

%Referring to \cite{cakoni_direct_2001,alessandrini2013stable}, we point that in order to prove the following lemma the critical results \eqref{eq:6.1} and \eqref{eq:6.2} are required, and we have established them for the obstacles of admissible class $\mathcal B$ uniformly. Therefore, the detailed proof of the lemma can be found in \cite{alessandrini2013stable}.

Lemma \ref{lem:63} establishes the relationship between the $L^2$-norm of the total wave field $u(\mathbf x)$ associated with \eqref{impedance problem}  on the boundary of a Lipschitz obstacle $K$ with the corresponding $L^2$-norm of $u$ in a related domain, which shall be used in proving Theorem \ref{th: eta}.

\begin{lem}\cite[Lemma~6.2]{alessandrini2013stable}\label{lem:63}
	Let $u(\mathbf x)$ be the weak solution of the impedance direct scattering problem~\eqref{impedance problem} associated with the obstacle $K$ with a Lipschitz boundary, where $K\Subset B_R$ with $R\in \mathbb R_+$. If the assumptions
	$$
	\|u(\mathbf x) \|_{C^\alpha(B_{R}\setminus \overline K) }\leq \mathcal C_R,
	$$
	and 
	$$
	\|u(\mathbf x) \|_{H^1(\partial K)}\leq \mathcal C_H,
	$$
	are satisfied, where $\mathcal C_R$ and $\mathcal C_H$ depend on the a-priori data  only, for every $\mathbf x_0 \in \partial K$ and 
	for every $0<r<r_0$ with $r_0\in \mathbb R_+$ satisfying $B_{r_0}(\mathbf x_0)\setminus \overline K \subset B_{R}\setminus \overline K$, then we have that
\begin{equation}\notag
	\left( \int_{B_{r}\left( \mathbf{x}_{0} \right)  \cap \partial K} \left\vert u\right\vert^{2}  \mathrm{d} \mathbf{x} \right)^{\beta }  \geq C\int_{B_{\frac{r}{2} }\left( \mathbf{x}_{0} \right)  \setminus \overline{K} } \left\vert u\right\vert^{2}  \mathrm{d} \mathbf{x}, 
\end{equation}
where $C > 0$, $0 < \beta < \frac{1}{2}$ are constants depending on the a-priori data only.
\end{lem}

Using the conclusions \eqref{eq:6.1} and \eqref{eq:6.2} for the admissible class $\mathcal B$  in Proposition \ref{uniform L2 bound},  Lemmas \ref{lem:62} and  \ref{lem:63}, we can obtain a uniform lower bound  \eqref{eq:62 lower} in the admissible class $\mathcal B$ for $L^2$-norm of  the total wave field $u(\mathbf x)$ associated with \eqref{impedance problem}  on the boundary of any obstacle $K \in \mathcal B$. 

%with respect to the admissible class $\mathcal B$. 
\begin{lem}\label{lem:uniform lower bound}
	Let $u(\mathbf x)$ be the solution of the impedance direct scattering problem~\eqref{impedance problem} associated with the obstacle $K\in \mathcal B$. For every $\mathbf x_0 \in \partial K$ and 
	for every $0<r<r_0$ with $r_0\in \mathbb R_+$ satisfying $B_{r_0}(\mathbf x_0)\setminus \overline K \subset B_{R}\setminus \overline K$, the following bound
\begin{equation}\label{eq:62 lower}
	\left( \int_{B_{r}\left( \mathbf{x}_{0} \right)  \cap \partial K} \left\vert u\right\vert^{2}  \mathrm{d} \mathbf{x} \right)^{\beta }  \geq \mathcal E_B, 
\end{equation}
holds, where  $0 < \beta < \frac{1}{2}$ and $\mathcal E_B$ are constants depending on the a-priori data only.

\end{lem}

\begin{proof}[Proof of Theorem~\ref{th: eta}]
From the result of Theorem~\ref{mian result}, under the assumption~\eqref{eq: new far error} we establish a stable determination \eqref{eq:stability function} for two admissible obstacles $K$, $K'\in \mathcal{B}$ with two non-zero impedance parameters $\eta$ and $\eta'$.
    
    Under the assumption \eqref{eq: new far error}, by the argument of Lemma~\ref{lem near error}, for $R>0$, there exist $\zeta \in \mathbb R_+$, and for any $\mathbf{x_0}\in B_{2R}\setminus B_R$ such that  $R+1+\zeta \leq \Vert \mathbf{x_0} \Vert \leq 2R$, one has the corresponding near-field error $\varepsilon_1$ such that 
    \begin{equation}\label{eq: new near error}
		\left\|u_{(K,\eta)}-u_{(K',\eta')} \right\|_{L^{\infty} (B_{\zeta}(\mathbf x_1))} \leq \varepsilon_1 , \ \ \   \forall \mathbf {x_1}\in B_{\Vert \mathbf{x_0} \Vert + \zeta} \setminus{\overline{B _{\Vert \mathbf{x_0} \Vert - \zeta}}}, 
	\end{equation}
    where $\varepsilon_1$ is defined by \eqref{eq: near error}.

    Let $u_{\left( K,\eta' \right)} $ be the total wave field to the impedance scattering \eqref{impedance problem} associated with  $(K,\eta')$. 
    Using the stability of the direct acoustic scattering problem (cf.\cite[Proposition 3.1]{rondi2008stable}), if the Hausdorff distance $\mathfrak h$ of the two admissible obstacles $K, K'$ is sufficiently small, then their corresponding near-field error have the following estimates 
    \begin{equation}\label{eq: direct problem stability}
    	\left\| u_{\left( K',\eta' \right)}  ({\mathbf{x} } ) - {u_{\left( K,\eta^{\prime } \right)}({\mathbf{x} } )} \right\|_{L^{\infty} (B_{\zeta}(\mathbf x_1))}\leq 
    	C_1\mathfrak h^\varsigma |\mathbf x|^{-1}, \quad 
    	\mathbf x\in B_{2R}\setminus B_R ,
    \end{equation}
    where the $0<\varsigma <1$ and $C_1>0$ are constants depending a-priori parameters.

    Combining \eqref{eq: new far error}, \eqref{eq: new near error},   \eqref{eq: direct problem stability}, with \eqref{eq: near error}, one has
    \begin{align}\label{eq: phi}
    	 &\left\| u_{\left( K,\eta \right)}   - {u_{\left( K,\eta^{\prime } \right)}} \right\|_{L^{\infty} (B_{\zeta}(\mathbf x_1))}\\ \nonumber
    	 &\leq  \left\| u_{\left( K,\eta \right)}   - {u_{\left( K',\eta^{\prime } \right)}} \right\|_{L^{\infty} (B_{\zeta}(\mathbf x_1))}
    	  + \left\| u_{\left( K',\eta' \right)}  - {u_{\left( K,\eta^{\prime } \right)}} \right\|_{L^{\infty} (B_{\zeta}(\mathbf x_1))}\\ \nonumber
    	 &\leq \varepsilon_1 + C_1 \mathfrak h^\varsigma |\mathbf x|^{-1} 
    	 \leq \exp(-C_a(-\ln \varepsilon)^{1/2})+ {\frac
    	 {C}{R}(\ln \ln (1/\varepsilon))^{-\varsigma \kappa}}:=\phi(\varepsilon),
    \end{align}
    where the constants $C$ and $\kappa$ are defined in Theorem~\ref{mian result}.

 %we assume that  there exists an open subset  $\Gamma_h $   of $\partial H \cap \partial K$ such that any point on $\Gamma_h $ can be connected to infinity through a path. 
    
 %In the following we only need to consider the case that $\mathfrak h>0$ since the case that $\mathfrak h\equiv 0$ can be proved similarly.      
 
% According to Remark \ref{rem:15}, we know that $u_{(K,\eta)}(\mathbf x)\in C(\overline{\mathbb R^2 \setminus K })$.   Without loss of generality,  we assume that $\mathbf x_0 \in \Gamma_h \Subset \partial H\cap \partial K$ such that 
 %$$
 %\mathbf x_0={\rm argmax}_{\mathbf x \in \Gamma  \setminus {\mathcal V}(K) }\{  |u_{\left( K,\eta \right)} (\mathbf x)|~\big |~ \Gamma \in \partial  H\cap \partial K \}
 %$$
 %where $H=\mathbb R^2 \setminus (K\cup K')$. 
 %As discussed in Remark \ref{rem:15}, we emphasize that $|u_{\left( K,\eta \right)} (\mathbf x_0)|>0.$  
 It is clear that any point on $\Gamma_h $ can be connected to infinity through a path. Hence,  we are preparing to propagate the near-field data \eqref{eq: phi}  to the boundary $\Gamma_h$.      Adopting a similar argument for Proposition~\ref{Lem:Boundary estimation},  we could obtain a local estimate on $\Gamma_h$:
    	\begin{align}\label{eq: new Cf}
	\sup_{\mathbf x\in\Gamma_h}|u_{(K,\eta)}(\mathbf x)-u_{(K,\eta')}(\mathbf x)| &\leq  C_f(\ln|\ln \phi(\varepsilon)|)^{-\alpha},\\ \sup_{\mathbf x\in\Gamma_h}\left|\nabla u_{(K,\eta)}(\mathbf x)-\nabla u_{(K,\eta')}(\mathbf x)\right | &\leq  C_f(\ln|\ln \phi(\varepsilon)|)^{-\alpha}.  \notag
	\end{align}
	By the impedance boundary condition, one has
	$$
	 \frac{\partial  u_{\left( K,\eta \right)}}{\partial\nu}+\eta  u_{\left( K,\eta \right)}=0 , \quad
	 \frac{\partial  u_{\left( K,\eta' \right)}}{\partial\nu}+\eta'
	 u_{\left( K,\eta' \right)}=0,\quad \mbox{on} \quad \Gamma_h
	$$  	
	 and
	\begin{equation}
		\label{eq: final}
	 \frac{\partial  u_{\left( K,\eta \right)}}{\partial\nu}-\frac{\partial  u_{\left( K,\eta' \right)}}{\partial\nu}
	 +(\eta -\eta^{\prime } )u_{(K,\eta)}+\eta^{\prime } \left[ u_{(K,\eta)}-u_{(K,\eta')}\right]  =0, \quad \mbox{on} \quad \Gamma_h. 
	\end{equation}

	Notice that we have assumed the surface impedance $\eta,\eta'\in\Xi_\mathcal B$ are two constant functions.
	Combining \eqref{eq: phi}, \eqref{eq: new Cf}, and \eqref{eq: final}, we have the following estimates 
	\begin{align}\label{eq:68}
		\left\Vert u_{(K,\eta )}(\mathbf{x} )(\eta -\eta')\right\Vert_{L^{2}\left( \Gamma_{h} \right)  }  \leq C_{2}(\ln |\ln \phi (\varepsilon )|)^{-\alpha },	
		\end{align}
	where the positive constant $C_2$ depends on a-prior parameters only.  	
	
	We apply Lemma~\ref{lem:uniform lower bound} and get that
	\begin{equation}\label{eq:63}
		\left( \int_{\Gamma_{h} } \left| u\right|^{2}  \mathrm{d} \sigma \right)^{\beta}  \geq \mathcal E_B, %C\left( \int_{B_{\frac{h}{2} }\left( \mathbf{y}_{+} \right)  } \left| u\right|^{2}  \mathrm{d} \mathbf{x}  \right),
	\end{equation}
	where $\mathcal E_B>0$, $0 < \beta < \frac{1}{2}$ are constants depending on the a-priori data only. The bound \eqref{eq:63} holds uniformly for obstacles in the admissible class $\mathcal B$. Combining \eqref{eq:68} with \eqref{eq:63}, one can derive \eqref{eq: final eta}, where  $C_P=\frac{C_2}{\mathcal E_B^{1/(2\beta)}}$.
	% and $\mathbf y_+$ is the midpoints of $\Gamma_h$. It is not difficult to find that we have the inequality \eqref{eq:63} holds for any $K\in\mathcal B$. Under the argument of Lemma~\ref{lem:62} and Lemma~\ref{lem:63}, we can conclude that there exists a constant $C_3$ such that \begin{equation} 	\left\Vert u\right\Vert_{L^{2}\left( \Gamma_{h} \right)  }  \geq C_{3} \end{equation} holds uniformly for obstacles in the admissible class $\mathcal B$.
	%The estimation of the impedance parameter is then obtained instantly by a brief calculation, where $C_P=\frac{C_2}{C_3}$.
\end{proof}

\section*{Acknowledgements}

The work of H. Diao is supported by National Key R\&D Program of China (No. 2020YFA\\0714102). The work of H. Liu is supported by the Hong Kong RGC General Research Funds (projects 12302919, 12301420 and 11300821), the NSFC/RGC Joint Research Fund (project  N\_CityU101/21), the France-Hong Kong ANR/RGC Joint Research Grant, A-HKBU203/19.

 %%%%%%

%---------
%\bibliographystyle{siam}
%\bibliography{ref.bib}

%-----------------------------------------------------------
\end{document}